\documentclass[10pt]{article}

\pdfoutput=1

\usepackage{amsfonts}
\usepackage{amssymb}
\usepackage{amsthm}
\usepackage{amsmath}
\usepackage{algorithm}
\usepackage[top=1in, bottom=1.25in, left=1.5in, right=1.5in]{geometry}
\usepackage{graphicx}
\usepackage{hyperref}
\usepackage{multirow}
\usepackage{subfig}
\usepackage{url}
\usepackage{breakurl}
\usepackage{verbatim}

\graphicspath{ {figures/} }

\hypersetup{
    colorlinks,
    linkcolor={blue},
    citecolor={red},
    urlcolor={green},
    breaklinks=true
}

\makeatletter
\def\namedlabel#1#2{\begingroup
  #2%
  \def\@currentlabel{#2}%
  \label{#1}\endgroup
}
\makeatother

\newcommand{\vertiii}[1]{{\left\vert\kern-0.25ex\left\vert\kern-0.25ex\left\vert #1 
    \right\vert\kern-0.25ex\right\vert\kern-0.25ex\right\vert}}

\newtheorem{theo}{Theorem}
\newtheorem{lem}{Lemma}
\newtheorem{prop}{Proposition}

\begin{document}

\title{Anisotropic residual based a posteriori mesh adaptation in 2D: element based approach}

\newcommand*\samethanks[1][\value{footnote}]{\footnotemark[#1]}
\author{Edward Boey
  \thanks{Department of Mathematics and Statistics, University of
      Ottawa, 585 King Edward Avenue, Ottawa, ON, Canada, K1N
      6N5}
  \thanks{Corresponding author email: eboey041@uottawa.ca}
  \and Yves Bourgault\samethanks[1]
  \and Thierry Giordano\samethanks[1]
}
\date{}

\renewcommand{\thefootnote}{\fnsymbol{footnote}} 

%
\maketitle

\begin{abstract}
  An element-based adaptation method is developed for an anisotropic a
  posteriori error estimator. The adaptation does not make use of a metric, but
  instead equidistributes the error over elements using local mesh
  modifications. Numerical results are reported, comparing with three popular
  anisotropic adaptation methods currently in use. It was found that the new
  method gives favourable results for controlling the energy norm of the error
  in terms of degrees of freedom at the cost of increased CPU usage.
  Additionally, we considered a new $L^2$ variant of the estimator. The
  estimator is shown to be conditionally equivalent to the exact $L^2$ error.
  We provide examples of adapted meshes with the $L^2$ estimator, and show that
  it gives greater control of the $L^2$ error compared with the original
  estimator.
\end{abstract}





\section{Introduction}
\label{sec:intro} 

In the last twenty years, anisotropic mesh adaptation has seen great activity.
Since the work of D'Azevedo and Simpson in \cite{dazsim89} and \cite{daz91} for
piecewise linear approximation of quadratic functions there has been a
significant amount of research dedicated to producing practical adaptation
procedures based on their results. In addition, there has been much software
written for the implementation, which either construct an entirely new mesh,
such as BAMG \cite{hec06}, BL2D \cite{laubou03}, GAMANIC3D \cite{geo03}, or
apply local modifications to a previous mesh, such as MEF++ \cite{gir}, MMG3D
\cite{dobthes}, YAMS \cite{yams}. The main idea they share in common is to
construct a non-Euclidean metric from the Hessian of the solution. We will refer
to them as Hessian adaptation methods, see for instance \cite{losala11a},
\cite{losala11b}, \cite{fregeo08}, \cite{borgeoheclausal97},
\cite{habdombouaitforval00}.

Residual a posteriori error estimation for elliptic equations has been around
for some time. In \cite{babrhe78a} and \cite{babrhe78b}, Babuska and Rheinboldt
introduced a local estimator, constructed entirely from the approximate
solution, that is globally equivalent to the energy norm of the error.
Numerical results showed that it was suitable for the purposes of mesh
adaptation by determining regions in which the mesh could be refined or
coarsened. While initially an entirely isotropic method, recently, the residual
method was modernized by the introduction of anisotropic interpolation estimates
from \cite{forper01}. Unlike classical results, the new estimates did not
require a minimum (or maximum) angle condition, and instead took into account
the geometric properties of the element. In \cite{pic03a} and \cite{forper03}
these interpolation results were combined with the standard a posteriori
estimates to drive mesh adaptation by constructing a metric. We will refer to
this method as the residual metric method. The method results in highly
anisotropic meshes, reducing the error by an order of magnitude compared to
isotropic methods \cite{pic03a}. Moreover, the procedure has been
successfully applied to a variety of nonlinear situations, including a
reaction-diffusion system to model solutal dendrites in \cite{burpic03} and
the Euler equations to model the supersonic flow over an aircraft in
\cite{boupicalalos09}.

\begin{sloppypar}
  Recent work in \cite{boiforfor12} demonstrates the potential advantages of
  element-based anisotropic mesh adaptation over the usual metric based mesh
  adaptation methods used so far. The error estimator they use is hierarchical:
  from a given approximate solution, they construct a higher-order, more
  accurate approximation. For the Hessian method it is necessary to take the
  absolute value of the eigenvalues of the Hessian, thus treating positive and
  negative curvature as essentially equal, while the distinction can be seen
  very clearly in meshes adapted with the hierarchical method. Further, the
  hierarchical estimator has the advantage that it can naturally be applied to
  finite elements of arbitrary order.
\end{sloppypar}

The primary goal of this paper is to introduce, and numerically assess, an
element-based adaptation approach to be used with the residual estimator from
\cite{pic03a}. We will refer to this method as the element-based residual
method. Motivation for implementing such a method includes avoiding the
additional steps involved in converting the estimator defined on elements, to a
metric defined on the nodes, during which information could be lost.
Additionally, we would like to attempt to mimic the success of the hierarchical
method. The adaptation will be implemented by interfacing the estimator with
the hierarchical adaptation code MEF++. We also introduce a variant of the
estimator for the $L^2$ norm error, which is shown to be reliable and efficient
under certain assumptions, and show that the estimator is also suitable for
anisotropic mesh adaptation. A secondary goal of the paper will be to provide a
comparative performance analysis between four different adaptation techniques:
element-based residual, metric based residual, Hessian, and hierarchical.

The outline of this paper is as follows: in Section 2 we introduce the model
problem and error estimator, as well as recall some results from the literature;
in Section 3 we discuss both the metric and element-based adaptation procedures;
in Section 4 we produce numerical results, validating the element-based method,
and comparing it with other anisotropic adaptation procedures.


\section{The estimator}
\label{sec:est}

We discuss the model problem and introduce a residual estimator. Main results
will be summarized from the literature. Full details can be found for instance
in \cite{forper01}, \cite{pic03a}, and \cite{micper06}.


\subsection{Model problem}
\label{subsec:ellipt}


Let $\Omega\subseteq\mathbb{R}^2$ be a bounded polygonal domain, with boundary
$\partial\Omega.$ Let $V=H^1(\Omega)$ and
$V_0=\{v\in H^1(\Omega):v|_{\partial\Omega}=0\}$. For
$g\in H^{1/2}(\partial\Omega)$, let
$V_g=\{v\in H^1(\Omega):v|_{\partial\Omega}=g\}$, which may be thought of as the
translation of $V_0$ by $g$. For $f\in L^2(\Omega),$ and a positive definite
matrix $A$, let $u\in V_g$ be the solution of the equation
\begin{equation}\label{eq-lap}
  \left\{\begin{array}{ll}
      -\text{div}( A\nabla u)=f,&\quad \text{in }\Omega, \\
      u=g,&\quad\text{on }\partial\Omega.
    \end{array}\right.
\end{equation}
Then $u$ is the solution to the variational equation
\begin{equation*}
  B(u,v)=F(v),\quad\quad\forall v\in V_0,
\end{equation*}
where
\begin{align*}
  B(u,v)&=\int_\Omega A\nabla u\cdot \nabla v\,\,\mathrm{d}x,\quad\quad u\in V_g,v\in V_0,\\
  F(v)&=\int_\Omega f v\,\,\mathrm{d}x,\quad\quad v\in V_0.
\end{align*}


For $h>0$, let $\mathcal{T}_h$ be a conformal triangulation of $\Omega$
consisting of triangles $K$ with diameter $h_K\leq h.$ Denote by $V_{h}$ the
finite element space of continuous, piecewise linear functions ($P_1$) on
$\mathcal{T}_h$ and $V_{h,0}$ the subspace of functions vanishing on
$\partial\Omega$. Let $g_h$ be a piecewise linear approximation of $g$ on
$\partial\Omega$ and let $V_{h,g}=\{v_h\in V_h:v_h|_{\partial\Omega}=g_h\}.$
Then the finite element approximation $u_h\in V_{h,g}$ of $u$ satisfies the
discrete variational equation
\begin{equation}\label{eq-var-disc}
  B(u_h,v_h)=F(v_h),\quad\quad\forall v_h\in V_{h,0}.
\end{equation}
For details on the finite element method for elliptic problems, see for instance
\cite{quaval97}.


\subsection{Anisotropic residual error estimator}
\label{subsec:interp}

\begin{sloppypar}
  Define the energy norm by $\vertiii{v}=B(v,v)^{1/2}$ for $v\in V$. The
  residual mesh adaptation procedure is based on controlling the energy norm of
  the discretization error $e_h=u-u_h$. The error estimator, which will be
  outlined below, combines information of the residual with anisotropic
  interpolation estimates.
\end{sloppypar}

Define the localized residual by
\begin{equation*}
  R_K(u_h)=f+\text{div}(A\nabla u_h),
\end{equation*}
where the divergence operator is local to $K$. The jump of the derivative for an
element $K$ with edges $e_i$ is defined by
\begin{equation*}
  r_K(u_h)=\sum_{i=1}^3[A\nabla u_h]_{e_i},
\end{equation*}
where the jump $[A \nabla v_h]_{e_i}$ over $e_i$ is defined as follows: denoting
the outward unit normal by $n_i$ and the adjacent element (if it exists) by
$K'$, then
\begin{equation*}
  [A \nabla v_h]_{e_i}=\left\{\begin{array}{ll}
      0,&\quad e_i\in\partial\Omega,\\
      A \nabla (v_h)|_{K}\cdot n_i - A \nabla (v_h)|_{K'}\cdot n_i,&\quad \text{otherwise}.
    \end{array}\right.
\end{equation*}

For a triangular element $K$, the anisotropic information comes from the affine
mapping $F_K:\hat{K}\rightarrow K$. The reference element $\hat{K}$ is taken to
be the equilateral triangle centred at the origin with vertices at the points
$(0,1),\,(\frac{-\sqrt{3}}{2},\frac{-1}{2}),\,(\frac{\sqrt{3}}{2},\frac{-1}{2})$.
The Jacobian $J_K$ of $F_K$ is non-degenerate, so the singular value
decomposition (SVD) $J_K=\mathcal{R}_K^T\Lambda_K\mathcal{R}_K \mathcal{Z}_K$
consists of orthogonal matrices $\mathcal{R}_K,\,\mathcal{Z}_K$, and positive
definite diagonal matrix $\Lambda_K$. The matrices $\mathcal{R}_K,\,\Lambda_K$
take the form
\begin{align*}
  \mathcal{R}_K=\begin{pmatrix}r_{1,K}^T\\r_{2,K}^T\end{pmatrix},
  \quad\quad\quad
  \Lambda_K=\begin{pmatrix}\lambda_{1,K} && 0\\0 && \lambda_{2,K}\end{pmatrix},
\end{align*}
where $\lambda_{1,K}\geq\lambda_{2,K}>0$, $r_{1,K},\,r_{2,K}$ are orthogonal
unit vectors. Geometrically, these eigenvalues and eigenvectors represent the
deformation of the unit ball in $\mathbb{R}^2$ to an ellipse with axes of length
$\lambda_{1,K},\,\lambda_{2,K}$ in directions $r_{1,K},\,r_{2,K}$ respectively.
Moreover, they represent $K$ in the sense that the ellipse circumscribes the
element.

Denote by $\Delta_K$ the patch of elements containing a vertex of $K.$ As noted
in \cite{micperpic03}, for the bounds for the quasi-interpolation operator to be
uniform, there must be an integer $\Gamma>0$ and a constant $C>0$ such that all
such patch satisfies $\text{card}(\Delta_K)\leq\Gamma$ (cardinality) and
$\text{diam}(F_K^{-1}(\Delta_K))\leq C$ (diameter). For $v\in V,$ define the
following ``Hessian'' type matrix:
\begin{equation}\label{eq-G}
  \tilde{G}_K(v)=
  \left(\int_{\Delta_K}\frac{\partial v}{\partial x_i}\frac{\partial v}{\partial x_j}\,\,\mathrm{d}x\right)_{i,j},
\end{equation}
and let
\begin{equation*}
  \tilde{\omega}_K(v)=(\lambda_{1,K}^2r_{1,K}^T \tilde{G}_K(v)r_{1,K}+\lambda_{2,K}^2r_{2,K}^T \tilde{G}_K(v)r_{2,K})^{1/2},
\end{equation*}
Finally, define
\begin{align}
  \hat{\eta}_K^2 
  &=\left(\|R_K(u_h)\|_{0,K} +
    \left(\frac{h_K}{\lambda_{1,K}\lambda_{2,K}}\right)^{1/2}\|r_K(u_h)\|_{0,\partial
      K}\right)\tilde{\omega}_K(e_h).\label{eq-est-aniso-norec}
\end{align}

\begin{theo}[\cite{pic03a}\cite{pic06}\cite{micper06}]
  \label{lem-est-aniso}
  There exist constants $C_{1,\hat{K}},\,C_{2,\hat{K}}>0$ such that
  \begin{align*}
    C_{1,\hat{K}} \sum_K\hat{\eta}_K^2\leq \vertiii{e_h}^2\leq C_{2,\hat{K}} \sum_K\hat{\eta}_K^2.   
  \end{align*}
\end{theo}

The upper and lower bounds in Theorem \ref{lem-est-aniso} still depend on the
unknown solution due to the $\tilde{\omega}_K(e_h)$ term. The approach taken in
\cite{pic03a} is to remove this dependency by using a gradient recovery operator
of the form $\Pi:V_h\rightarrow V_h\oplus V_h.$ The operator should be
super-convergent in the sense that $\Pi(u_h)$ converges to $\nabla u$ faster
than $\nabla u_h$, at least of order $1+\epsilon$ where $\epsilon>0$. For full
details on the derivation of upper and lower bounds using super-convergence
assumptions, see \cite{micper06}. Therefore, for the remainder we replace
$\tilde{G}_K(e_h)$ by
\begin{equation*}
  G_K(u_h)=\left(\int_K\left(\frac{\partial u_h}{\partial x_i} - \Pi(u_h)_i\right)\left(\frac{\partial u_h}{\partial x_j} - \Pi(u_h)_j\right)\,\mathrm{d}x\right)_{i,j},
\end{equation*}
and $\tilde{\omega}_K(e_h)$ by
$\omega_K(u_h)=(\lambda_{1,K}^2r_{1,K}^T
G_K(u_h)r_{1,K}+\lambda_{2,K}^2r_{2,K}^T G_K(u_h)r_{2,K})^{1/2}$ and
the estimator becomes
\begin{align}
  \eta_K^2 
  &=\left(\|R_K(u_h)\|_{0,K} +
    \left(\frac{h_K}{\lambda_{1,K}\lambda_{2,K}}\right)^{1/2}\|r_K(u_h)\|_{0,\partial
      K}\right)\omega_K(u_h).\label{eq-est-aniso}
\end{align}
Note furthermore, that the integral for the matrix $\tilde{G}_K(e_h)$ is taken
on the patch $\Delta_K$ while that for $G_K(u_h)$ is taken only on the element
$K$. We found simplification works in practice and greatly reduces the
computational complexity of the estimator, and has been used for instance in
\cite{pic03a}, \cite{lozpicpra09}.


\subsection{Gradient recovery}
\label{subsec:recov}

Here we discuss briefly our choice of gradient recovery method. A popular choice
is the simplified Zienkiewicz-Zhu (ZZ) operator, see \cite{rod94}, which
generally performs very well. For instance, on certain regular meshes (parallel)
it is asymptotically exact. Moreover, despite the fact that it cannot be proven
to be super-convergent for non-regular meshes, in practice superconvergence has
been observed for adapted meshes, as in \cite{pic03a} and \cite{micper06}.

An improved method is proposed by Zhang and Naga in \cite{zhanag05}. The main
idea is that for each node, one fits the solution values to a higher-order
polynomial on a surrounding patch, the fit being obtained in a least-square
sense. The value of the recovered gradient at the node is obtained by taking the
gradient of the higher-order polynomial. They prove that the method is
super-convergent for any regular mesh pattern, including situations where the ZZ
estimator is not, such as the chevron pattern \cite{rod94}. In addition,
while the ZZ estimator only preserves polynomials of degree 1, their method can
be extended to higher-order elements.

In this paper we have chosen to use the recovery method of Zhang and Naga due to
an observed increase in performance. We remark that the usual justification of
use of the ZZ estimator is its low cost. However, the gradient recovery is only
computed once at the start of each iteration of the adaptation loop. As it
turns out in our case, calculation of the Zhang/Naga gradient recovery accounted
for less than $0.5$\% of the total CPU time.


\section{Adaptive procedure}
\label{sec:meth}

In this section, we describe the four mesh adaptation methods that will be
compared in Section \ref{sec:numer}, starting with the new, element-based
adaptation procedure for the residual estimator $\eta_K.$ The section concludes
with a discussion of the control of the $L^2$ norm error vs. the $H^1$
seminorm error.


\subsection{Element-based adaptation}
\label{subsec:elem}

\begin{sloppypar}
  By Theorem \ref{lem-est-aniso}, the estimator
  $\eta=\left(\sum_K\eta_K^2\right)^{1/2}$ is globally equivalent to the energy
  norm of the error $\vertiii{e_h}$. Given an error tolerance $TOL>0,$ the
  adaptation algorithm will attempt to control the error so that
  $\eta\approx TOL$. Moreover, the mesh should have the least possible number of
  elements $N_T$. Therefore, the primary goal of the adaptation algorithm is to
  equidistribute the estimated error by asking that every element $K$ satisfies
  $\eta_K^2\approx\frac{TOL^2}{N_T}$. From an initial calculation of $\eta_K$ we
  adapt the mesh by performing the following local mesh modifications: edge
  refinement, edge swapping, node removal, and node displacement. For a complete
  description of local mesh modifications, see for instance \cite{boiforfor12},
  \cite{dasforper15}, \cite{fregeo08}, \cite{habdombouaitforval00}.
\end{sloppypar}


For convenience we define two element patches, which will be referred to
frequently while discussing the local modifications. For an edge $e$, the patch
$\Delta_e$ will denote the patch of elements containing $e$. Similarly, for
a vertex $p$, the patch $\Delta_p$ will denote the patch of elements containing
$p$.


\subsubsection{Edge Refinement}
\label{subsubsec:ref}

Edge refinement is used to decrease the level of error where it is too large.
The candidate edges for refinement are those belonging to an element $K$ for
which $\eta_K^2> 1.5\frac{TOL^2}{N_T}.$ For such an edge $e$ with associated
edge patch $\Delta_e$, denote by $\Delta_e'$ the resulting patch after refining
$e$, and suppose they have respectively $N_{T,e},$ $N_{T,e'}$ elements. Denote
respectively by $\eta_{\Delta_e}^2$ and $\eta_{\Delta_e'}^2$ the error on the
patch before and after refinement. The refinement is accepted if the new error
is closer to the goal in the following sense:
\begin{align}\label{ineq-ref}
  \left|\frac{\eta_{\Delta_e'}^2}{N_{T,e'}}-\frac{TOL^2}{N_{T}}\right| < \left|\frac{\eta_{\Delta_e}^2}{N_{T,e}}-\frac{TOL^2}{N_{T}}\right|.
\end{align}


\subsubsection{Edge swapping}
\label{subsubsec:swap}

Edge swapping is used to minimize the error without changing the number of
elements. For an internal edge $e$, consider the edge patch $\Delta_e$, test the
reconnection of the edge, and denote this patch $\Delta_{e'}$. Note that it may
be geometrically impossible to swap an edge, for instance if the patch is not
convex, or degenerates to a triangle. Edge swapping is performed if the global
error decreases. At first, one might try swapping if the following criterion
holds:
\begin{equation*}
  \sum_{K'\in\Delta_{e'}}\eta_{K'}^2 < \sum_{K\in\Delta_e}\eta_K^2.
\end{equation*}
However, for the residual estimator, the above criteria is not enough, and we
had to enlarge the patch as in Figure \ref{fig:ext-patch}. Note that swapping
the edge changes the normal jump of the derivative for elements adjacent to
$\Delta_e$. Including these elements in the error calculation means that we
have included all elements for which $\eta_K$ is changed by swapping, so that if
the error decreases on the patch, then in fact the error will have decreased
globally.

\begin{figure}[H]
  \centering
  \includegraphics[width=0.35\textwidth]{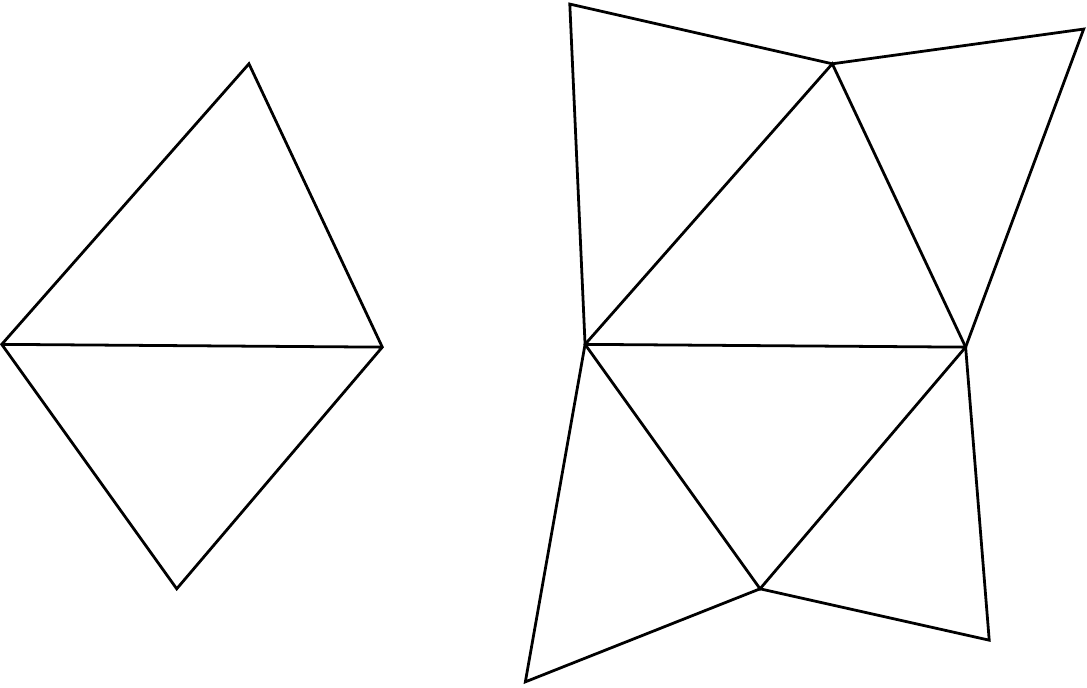}
  \caption{Extension of the edge patch for edge swapping.}
  \label{fig:ext-patch}
\end{figure}

The edges are stored in an ordered list, and edge swapping is carried out by
looping over this list and checking all internal edges. When an edge is
swapped, the new edge is placed at the end of the list, so will be considered
for swapping again. After the list is exhausted, if any edges were
swapped the entire procedure will be repeated.

We remark that the loop will in fact terminate. There are only finitely many
ways to reconfigure the mesh by swapping, and the edges are only swapped if the
global error decreases. Furthermore, because no interpolation of the solution
takes place during edge swapping, the map from edge configuration to global
error calculation is well-defined, so it is not possible to arrive at a previous
configuration but with smaller error. Also, as noted on p. 1339 of
\cite{pic03a}, the choice of reference element means that the contribution of
the SVD will not depend on a reordering of the nodes.


\subsubsection{Node removal}
\label{subsubsec:rem}

Node removal is used to reduce the number of mesh elements where possible,
particularly where the error is small. Node removal consists in removing a node
$p$ from the mesh, as well as the patch of elements, $\Delta_p$, attached to the
node. The resulting ``hole'' then is remeshed, and we will call the resulting
patch $\Delta_p'$. The initial choice of remeshing is not important because the
optimal choice will be determined by edge swapping. One compares the error
before and after the procedure, denoted $\eta_{\Delta_p},$ $\eta_{\Delta_p'}$,
and the node is accepted for removal if the following analogue to
(\ref{ineq-ref}) holds
\begin{align}\label{ineq-deref}
  \left|\frac{\eta_{\Delta_p'}^2}{N_{T,p'}}-\frac{TOL^2}{N_{T}}\right| <
  \left|\frac{\eta_{\Delta_p}^2}{N_{T,p}}-\frac{TOL^2}{N_{T}}\right|.
\end{align}


\subsubsection{Node displacement}
\label{subsubsec:disp}

The goal of node displacement is to equidistribute the error over the mesh
elements. Node displacement is applied to each vertex $p$ to determine the
optimal position of the vertex within the vertex patch $\Delta_p$. Note that
this patch might not be convex, so care has to be taken to avoid overlapping
elements. We consider the value of the error on the elements as a discrete
distribution, and find the position within the patch which minimizes the
variance: $\min_p\left(\text{Var}_{K\in\Delta_p}\{\eta_K^2\}\right).$ No attempt
is made to solve the minimization problem fully for each vertex, but only to
find an approximate solution with one iteration of a gradient recovery method.
Computing the full solution could be costly, and moreover might not even be
possible depending on the shape of the function being minimized. Instead, one
applies several iterations of the global node displacement procedure. As we
will see in Section \ref{subsec:hier}, node movement minimizes a different
function in the hierarchical method from \cite{boiforfor12}.


Edge refinement and node removal work towards achieving the error tolerance,
node movement ``smooths'' the mesh by equidistributing the error, while edge
swapping minimizes the error.


%
  

After a mesh operation is performed the error estimator needs to be
recalculated. First we interpolate the continuous data, which in this case is
$u_h$ and the recovered gradient $\Pi(u_h).$ The discontinuous data needs to be
recalculated on each element, i.e. the singular value decomposition,
discontinuous gradient, jump of the derivative, and the residual. Additionally,
after each operation is performed there is a check to ensure degenerate elements
were not produced.

All numerical results are produced with MEF++. The hierarchical estimator
adaptation driver was used, described in \cite{boiforfor12}, suitably adjusted.


\subsection{Hierarchical}
\label{subsec:hier}

We summarize the ideas from \cite{boiforfor12}. Given a $P_k$ approximation
$u_{h,k},$ construct a higher-order solution $P_{k+1}$, $\tilde{u}_{h,k+1}$,
which is supposed to be more accurate. From this, one obtains an approximation
of the error
\begin{equation}\label{eq-err-hier}
  e_h\approx \tilde{u}_{h,k+1}-u_{h,k}.
\end{equation}
Taking $k=1$, and the barycentric representation of the element $K$ by
$u_{h,1}|_K=u_1\lambda_1 + u_2\lambda_2 + u_3\lambda_3$,
one builds $\tilde{u}_{h,2}$ in the ``hierarchical'' basis
\begin{equation*}
  \tilde{u}_{h,2}|_K=u_1\lambda_1 + u_2\lambda_2 + u_3\lambda_3 + 4(e_1\lambda_1\lambda_2 +
  e_2\lambda_1\lambda_3 + e_3\lambda_2\lambda_3),
\end{equation*}
where $e_i$ denotes the mid-edge values. Taking the Zhang-Naga (or any other
sufficiently accurate) recovered gradient $\Pi(u_{h,1})=(\Pi(u_{h,1})_1,\Pi(u_{h,1})_2)$, the
mid-edge values are found by enforcing consistency between the Hessian of
$\tilde{u}_{h,2}$ and the derivatives of $\Pi(u_h)$:
\begin{align*}
  &\frac{\partial^2\tilde{u}_{h,2}}{\partial x_{1}^2} 
  = \frac{\partial\Pi(u_{h,1})_1}{\partial x_1},\quad\quad
  \frac{\partial^2\tilde{u}_{h,2}}{\partial x_{2}^2}
  = \frac{\partial\Pi(u_{h,1})_2}{\partial x_2},\\
  &\frac{\partial^2\tilde{u}_{h,2}}{\partial x_{1}\partial x_{2}} 
  = \frac{1}{2}\left(\frac{\partial\Pi(u_{h,1})_1}{\partial x_2} + \frac{\partial\Pi(u_{h,1})_2}{\partial x_1}\right).
\end{align*}
Having computed the higher-order solution, the adaptation process follows
similarly to that used in Section \ref{subsec:elem}. But since
(\ref{eq-err-hier}) gives a direct representation of the error field, one has
considerably more freedom in how to calculate the error on each element. The
choice in \cite{boiforfor12}, and as implemented by the authors in MEF++, is to
target the global error in the $L^2$ norm. The operations of edge refinement
and node removal will be used to achieve a global level of error, while node
displacement and edge swapping are used to locally equidistribute the error by
minimizing the gradient of the error, i.e. the $H^1$-seminorm error.

In this paper, we only consider hierarchical adaptation for $P_1$ finite
elements for the sake of comparison. However, note that (\ref{eq-err-hier}) is
quite general, and it is very easy to generalize these ideas to higher-order
finite elements. The hierarchical method has been successfully applied to $P_2$
finite elements in \cite{boithes}.


\subsection{Metric adaptation}
\label{subsec:met}

Currently, the most popular anisotropic mesh adaptation methods in use are
metric based. Here, the main idea is to control the edge length in a Riemannian
metric. For a planar domain $\Omega$, an inner product is given by a set
$\{\mathcal{M}(x)|x\in\Omega\}$ of $2\times 2$ positive definite matrices. In
practice we only have a discrete approximation, consisting of a metric defined
at the nodes of the mesh, the values at other points being obtained by
interpolation \cite{borgeoheclausal97}. For an edge $e=PQ$, the edge
length is given by
\begin{equation}\label{eq-met}
  |e|_\mathcal{M}=\int_0^1\sqrt{e^T\mathcal{M}(P+te)e}\,\,\mathrm{d}t.
\end{equation}
The goal of a metric based adaptation algorithm will be to generate meshes which
are ``unit'' with respect to the metric. For 2D meshes this simply means that,
up to some tolerance, the edges have unit length.

The metric adaptation will be done using MEF++, applying the same mesh
modification operations discussed Section \ref{subsec:elem}. The goal of edge
refinement and node removal is to achieve unit edge length, while the second two
locally equidistribute the error. More precisely, edge swapping applies a
non-Euclidean variant of the classical Delaunay edge swapping criterion to
maximize the minimum angle. For node movement, the edges attached to a node are
seen as a network of springs with stiffness proportional to metric edge length,
and the goal is to minimize the ``energy'' of the system. For full details see
for instance \cite{habdombouaitforval00}.

%


\subsection{Residual metric based}
\label{subsec:res-met}

Now we describe how the residual estimator introduced in Section \ref{sec:est}
can be used to define a metric. There exist at least two approaches used in the
literature, both following similar principles. The one we will use is that from
\cite{micper06} since it resulted in unit meshes in only a few iterations. The
metric is constructed locally for the element $K$ by finding the shape of a new
element $K_{new}$ which minimizes $\eta_{K_{new}}$ up to a fixed area. From
\cite{micper06}, Proposition 26, the minimizing shape is given by
$\tilde{r}_{1,K}=p_2,\,\tilde{r}_{2,K}=p_1,$ and
$\tilde{s}_K=\sqrt{\frac{\alpha_{1,K}}{\alpha_{2,K}}},$ where
$\alpha_{1,K}\geq\alpha_{2,K}>0$ and $p_1,\,p_2$ are respectively the
eigenvalues and eigenvectors of the normalized matrix $\frac{G_K}{|K|}.$ Under
these conditions, one obtains a simple relation for the error
\cite[p. 826]{micper06}. Imposing $\eta_K\equiv\tau$, where $\tau>$ is the local
error tolerance, one determines the area from this relation, and then easily
recovers the optimal values $\tilde{\lambda}_{1,K}$,
$\tilde{\lambda}_{2,K}$. Finally, defining $\tilde{R}_K,\,\tilde{\Lambda}_K$ in
the obvious way, the metric on $K$ is then given by
$\tilde{\mathcal{M}}_K=\tilde{R}_K^T\tilde{\Lambda}_K^{-2}\tilde{R}_K.$ We
remark that the mesh adaptation software used for this paper requires the metric
to be defined on vertices, and for this purpose we apply metric
intersection. For a vertex $p$, the metric $\tilde{\mathcal{M}}_p$ will be
defined as the intersection of all the metrics $\tilde{\mathcal{M}}_K$ over the
patch $\Delta_p.$ For details on metric intersection, see for instance
\cite{fregeo08}. Note that, alternatively to metric intersection, we found that
a simple averaging procedure gave satisfactory results.


\subsection{Hessian}
\label{subsec:hess}

The Hessian metric approach introduced here follows the expositions from
\cite{borgeoheclausal97}, \cite{habdombouaitforval00}. The $P_1$ interpolation
error $e_h^I=u-I_h(u)$ of a function $u$ on an edge $\ell=[x_i,x_j]$ satisfies
\begin{equation*}
  |e_h^I|_{L^\infty(\ell)}=\frac{|\ell|^2}{8}\left|\frac{\mathrm{d}^2 u}{\mathrm{d} x^2 }(\xi)\right|,
\end{equation*}
where $\xi$ is some point in $\ell.$ The error on the edge can then approximated
using the end-points of the interval
\begin{equation}\label{eq-met-2D}
  |e_h^I|_{L^\infty(\ell)}\approx\frac{1}{2}\left(\ell^T |H(x_i)|\ell
    +\ell^T |H(x_j)|\ell\right),
\end{equation}
where $H(x)$ is the Hessian of $u$ at $x,$ and $|H(x)|$ is the positive
semi-definite matrix obtained by taking the absolute value of the eigenvalues of
the symmetric matrix $H(x)$. Fixing an error level $e_D,$ defining the metric
$\mathcal{M}_{x}=\frac{1}{8\cdot e_D}|H(x)|$, and computing the edge length from
(\ref{eq-met}) with the trapezoid rule give $|\ell|_\mathcal{M}=1$ precisely
when the approximate error (\ref{eq-met-2D}) is equal to $e_D.$ For a slightly
different approach to Hessian methods, see \cite{losala11a} and \cite{losala11b}.

Note that (\ref{eq-met-2D}) depends on the unknown Hessian of the solution. One
may obtain a piecewise linear approximation of the Hessian from the computed
solution, for instance, with the least square fitting method used in
\cite{belforcha04}. 


\subsection{$L^2$ error vs. $H^1$ seminorm error estimation}
\label{subsec:H1-vs-L2}




While the residual estimator targets the energy norm of the error, it is also
interesting to determine whether we can expect to control the $L^2$ norm of the
error. Recall that for elliptic problems such as (\ref{eq-lap}), if a set of
meshes uniformly satisfies the minimum (or maximum) angle condition for some
angle $\theta>0,$ then there exists $C_1>0$ such that for any such mesh with
maximum edge length $h>0$,
\begin{equation}\label{order1}
  |e_h|_{1,\Omega}\leq C_1h.
\end{equation}
Furthermore, the Aubin-Nitsche Lemma states that there is a $C_2>0$ such that
the $L^2$ error satisfies
\begin{equation}\label{aubin}
  \|e_h\|_{0,\Omega}\leq C_2h|e_h|_{1,\Omega}.
\end{equation}
Combining (\ref{order1}) with (\ref{aubin}) one concludes that there is a
$C_3>0$ such that
\begin{equation}\label{order2}
  \|e_h\|_{0,\Omega}\leq C_3h^2.
\end{equation}
Therefore, if we adapt the mesh to control the $H^1$ seminorm error, which is
the case for the residual estimator, we expect higher-order convergence for the
$L^2$ error coming from an upper bound similar to (\ref{order2}).

In the absence of an Aubin-Nitsche Lemma in the context of anisotropic meshes,
we attempt to find a comparison at the element level between the $L^2$ and
energy norm of the error. The ultimate goal of the analysis is to derive an
$L^2$ norm variant of the residual estimator $\eta_K$, which could be
used for mesh adaptation.

To motivate our results, we first recall two existing a posteriori $L^2$ error
estimators. The first, in the isotropic setting is of the form
$\eta_{L^2(K)}=h_K\eta_{H^1(K)},$ where $\eta_{H^1(K)}$ is an estimator for the
energy norm, see (3.20) and (3.31) from \cite{ainode97} and 5.1 from
\cite{carver99}). Similarly, the authors in \cite{kunver00} derived an
anisotropic estimate of the form $\tilde{\eta}_E=(h_{\min,E})\eta_E$, where
$\eta_E,\,\tilde{\eta}_E$ are respectively estimates for the $H^1$ seminorm and
$L^2$ norm of the error for the edge $E$, and where $h_{\min,E}$ plays an
analogous role to $\lambda_{2,K}$ in the current setting. These observations
lead us to propose the following candidate for an $L^2$ error estimator:
\begin{equation}\label{eqn:eta-scaled}
  \tilde{\eta}_K=\lambda_{2,K}\eta_K,
\end{equation}
where $\eta_K$ is the energy norm estimator (\ref{eq-est-aniso}). In what
follows we present some partial results towards the reliability and efficiency
of this estimator. 

We make the following strong assumption on the equivalence of the energy norm
error: there exist $C_1,\,C_2>0$ such that for every element $K$,
\begin{align}\label{eta-equiv-elem}
  C_1\eta_K\leq|e_h|_{1,K}\leq C_2\eta_K.
\end{align}
Given this assumption, the strategy will be to relate the $L^2$ error and energy
norm locally. Note that in the literature, the upper bound in
(\ref{eta-equiv-elem}) only appears globally (for the entire domain $\Omega$),
while the lower bound holds on a patch related to a quasi-interpolation
operator, see \cite[Propositions 16, 21]{micper06}. Numerical results
in Section \ref{subsubsec:res-elem-adapt} suggest that provided the mesh is not
too coarse, the inequality holds with $C_1=1$ and $C_2=10$, see
Figure \ref{func3-errdist}.







We begin with a technical lemma. In what follows we let $W_h$ denote the
space of continuous piecewise quadratic functions on $\mathcal{T}_h$. We note,
however, that $W_h$ could be replaced by a higher-order finite element space,
with different constants for the inequalities.

\begin{lem}\label{lem-l2-lower} Let $v_h\in V_h$.
  \begin{enumerate}
    \item There exists $C_{\hat{K}}>0$ depending only on the reference element
      $\hat{K}$ such that for all $w_h\in W_h$ and $K\in\mathcal{T}_h$,
  \begin{align}\label{eq-l2-lower} 
    \|v_h-w_h\|_{0,K}\geq
    &\,C_{\hat{K}}\lambda_{2,K}|v_h-w_h|_{1,K}.
  \end{align}
  \item Suppose that $w_h\in W_h$ and $C_{\hat{K},1}>0$ such that
    for all $K\in\mathcal{T}_h$
  \begin{align}\label{eqn:direct-equi}
    \lambda_{1,K}\|\nabla (v_h - w_h)\cdot r_{1,K}\|_{0,K}
    &\leq C_{\hat{K},1}\lambda_{2,K}\|\nabla (v_h-w_h)\cdot r_{2,K}\|_{0,K}.
  \end{align}
  Then there exists $C_{\hat{K}}>0$ depending only on the reference element
  $\hat{K}$ and $C_{\hat{K},1}$ such that for all $K\in\mathcal{T}_h$,
  \begin{align}\label{eq-l2-upper}
    \|v_h-w_h\|_{0,K}
    &\leq C_{\hat{K},2}\lambda_{2,K}|v_h-w_h|_{1,K}.
  \end{align}
  \end{enumerate}
\end{lem}

\begin{proof} 
  By \cite{forper01}, Lemma 2.2 we have
  \begin{align*} 
    |w_h-v_h|_{1,K} 
    &\leq\left(\frac{\lambda_{1,K}}{\lambda_{2,K}}\right)^{1/2}|\hat{w}_h-\hat{v}_h|_{1,\hat{K}},
  \end{align*} 
  where $\hat{w}=w\circ F_K$ for a function $w$ on $K.$ Since $P_2(\hat{K})$ is finite
  dimensional, there exist positive constants $\tilde{C}_1,\,\tilde{C}_2$ such that
  \begin{align*} 
    \tilde{C}_1\|\hat{w}\|_{0,\hat{K}}\leq|\hat{w}|_{1,\hat{K}}\leq\tilde{C}_2\|\hat{w}\|_{0,\hat{K}},\quad\quad\forall \hat{w}\in
    P_2(\hat{K}).
  \end{align*} 
  Therefore,
  \begin{align} 
    |w_h-v_h|_{1,K} 
    &\leq \tilde{C}_2\left(\frac{\lambda_{1,K}}{\lambda_{2,K}}\right)^{1/2}\|\hat{w}_h-\hat{v}_h\|_{0,\hat{K}}\nonumber\\
    &=\tilde{C}_2\left(\frac{1}{\lambda_{1,K}\lambda_{2,K}}\right)^{1/2}
    \left(\frac{\lambda_{1,K}}{\lambda_{2,K}}\right)^{1/2}
    \|w_h-v_h\|_{0,K}\nonumber\\
    &=\frac{\tilde{C}_2}{\lambda_{2,K}}\|w_h-v_h\|_{0,K},\label{low-finite-equiv}
  \end{align} 
  and (\ref{eq-l2-lower}) follows. For (\ref{eq-l2-upper}), we have
  \begin{align*}
    \|v_h-w_h\|_{0,K}
    &=\left(\lambda_{1,K}\lambda_{2,K}\right)^{1/2}\|\hat{u}_h-\hat{u}_{h,2}\|_{0,\hat{K}}\\
    &\leq \tilde{C}_1\left(\lambda_{1,K}\lambda_{2,K}\right)^{1/2}|\hat{u}_h-\hat{u}_{h,2}|_{1,\hat{K}}.
  \end{align*}
  Applying \cite{forper01}, equation (17) to the right side of the inequality,
  and applying assumption (\ref{eqn:direct-equi}),
  \begin{align*}
    &\|v_h-w_h\|_{0,K}\\
    &\leq \tilde{C}_1
    \left(\lambda_{1,K}^2\|\nabla(v_h-w_h)\cdot r_{1,K}\|_{0,K}^2
      + \lambda_{2,K}^2\|\nabla(v_h-w_h)\cdot r_{2,K}\|_{0,K}^2\right)^{1/2}\\
    &\leq \tilde{C}_1C_{\hat{K},1}\lambda_{2,K}|v_h-w_h|_{1,K}.
  \end{align*}

\end{proof}

In the present situation we take $v_h=u_h.$ To apply Lemma \ref{lem-l2-lower} in
a meaningful way, we would like to find functions $\{w_h\in W_h\}_h$ that
converge to $u$ faster than $\{u_h\in V_h\}_h$ and that moreover satisfy
(\ref{eqn:direct-equi}) uniformly. Let $g_h=\Pi_h(u_h)$ denote the recovered
gradient, which is assumed to be superconvergent. In the literature, the
adaptive algorithm is designed to achieve the equality
\begin{equation*}
  \lambda_{1,K}\|(\nabla u_h-g_h) \cdot r_{1,K}\|_{0,K}=\lambda_{2,K}\|(\nabla
  u_h-g_h)\cdot r_{2,K}\|_{0,K},
\end{equation*}
In the context of \cite{micper06}, this equality means that $\eta_K$ has been
minimized with respect to the choice of $r_{1,K},\,r_{2,K}$ and aspect ratio
$s_K$. The adaptive algorithm discussed in Section \ref{subsubsec:first-err-l2}
will ensure that the equality holds by minimizing of $\eta_K$ with edge swapping
and node movement. In general, $g_h$ is not the gradient of a function in
$W_h$. Instead, we take $\tilde{u}_{h,2}\in W_h$ to be the hierarchical
reconstruction introduced in \cite{boiforfor12}, which in practice provides a
higher-order approximation to $u$, for instance \cite[Figure
17]{boiforfor12}. Additionally, it will be assumed that $\nabla \tilde{u}_{h,2}$
and $g_h$ are close enough so that (\ref{eqn:direct-equi}) holds with
$w_h=\tilde{u}_{h,2}.$

\begin{prop}\label{prop-l2}
  With the notation and assumptions of the preceding paragraph, there exist
  positive constants $C_{\hat{K},1},\,C_{\hat{K},2}$ such that for all
  $K\in\mathcal{T}_h$,
  \begin{align}\label{ineq-h1-l2-low}
    \|e_h\|_{0,K}
    &\geq C_{\hat{K},1}\left(\lambda_{2,K}|e_h|_{1,K}
    - \|u-\tilde{u}_{h,2}\|_{0,K} - \lambda_{2,K}|u-\tilde{u}_{h,2}|_{1,K}\right),
  \end{align}
  and
  \begin{align}\label{ineq-h1-l2-up}
    \|e_h\|_{0,K}
    &\leq C_{\hat{K},2}\left(\lambda_{2,K}|e_h|_{1,K} 
    + \|u-\tilde{u}_{h,2}\|_{0,K}
    + \lambda_{2,K}|u-\tilde{u}_{h,2}|_{1,K}\right)
  \end{align}

\end{prop}

\begin{proof}
  This follows directly from Lemma \ref{lem-l2-lower} and straightforward
  triangle inequality arguments.
\end{proof}


Finally, if we apply the superconvergence assumptions on $\tilde{u}_{h,2}$ and $\nabla
\tilde{u}_{h,2}$, and the strong energy norm error assumption (\ref{eta-equiv-elem}), we
conjecture that there exists $C_{\hat{K},1},\,C_{\hat{K},2}>0$ such that for
every element $K,$ up to the addition of higher-order terms, the $L^2$ norm
error satisfies
\begin{align}\label{eta-scale-equiv}
  C_{\hat{K},1}\tilde{\eta}_K
  \leq \|e_h\|_{0,K}
  \leq C_{\hat{K},2}\tilde{\eta}_K.
\end{align}
This local estimate will be verified numerically in the following section.

%



\section{Numerical results}
\label{sec:numer}

In this section we provide numerical validation for the new, element-based
adaptation method for the residual estimator. The full adaptation loop is given
in Algorithm \ref{loop}. The first test case will begin with an illustration of
the convergence of the loop for the element-based residual method introduced in
this paper, the notion of convergence to be made more precise in the following
section. Next, we will assess how well the method performs in achieving the
goal of equidistributing the error over the elements of the mesh. Since what we
are really interested in is controlling the actual error, the analysis will
include an element-level comparison of the estimated versus exact error.
Following will be a numerical validation of the $L^2$ error control results from
Section \ref{subsec:H1-vs-L2}. Finally, the remainder of the section will be
devoted to the comparison of the adaptation methods outlined in Section
\ref{sec:meth}.

\begin{algorithm}[h]
  \caption{Solution-adaptation loop}
  \label{loop}
  \begin{enumerate}
    \item[\namedlabel{itm:solve}{1}] Compute the solution and error estimator on
      the current mesh.
    \item[\namedlabel{itm:adap-rep}{2}] Adapt the current mesh by performing the
      following loop one or more times:
      \begin{enumerate}
        \item[\namedlabel{itm:ref}{(a)}] Refine edges where the error is too
          large.
        \item[\namedlabel{itm:smooth}{(b)}] Minimize the error by swapping edges
          until the algorithm terminates, then equidistribute the error by
          applying node displacement. Repeat the procedure one or more times.
        \item[\namedlabel{itm:deref}{(c)}] Remove nodes where the error is too
          small, or when the impact on the error is minimal.
        \item[\namedlabel{itm:sm-rep}{(d)}] Apply
          \ref{itm:adap-rep}\ref{itm:smooth}.
      \end{enumerate}
  \end{enumerate}
\end{algorithm}


\subsection{First test case}
\label{subsec:first}


We consider the problem (\ref{eq-lap}) using $A=I,$ with domain
$\Omega=(0,1)\times (0,1),$ and $f,\,g$ chosen so that the exact solution is
\begin{equation*}
  u_1(x,y)=4(1-e^{-100x}-x(1-e^{-100}))y(1-y).
\end{equation*}
Due to the boundary layer near $x=0$, this function can be used to check the
anisotropy of an error estimator and adaptation method, and appears for instance
in \cite{forper01} and \cite{pic03a}.


\subsubsection{Assessment of the residual element-based method}
\label{subsubsec:res-elem-adapt}


\paragraph{Convergence of the adaptation loop.}
\label{para:first-conv}

\begin{figure}[h]
  \centering
  \subfloat{
    \label{func3-mesh-init}
    \includegraphics[width=0.3\textwidth]{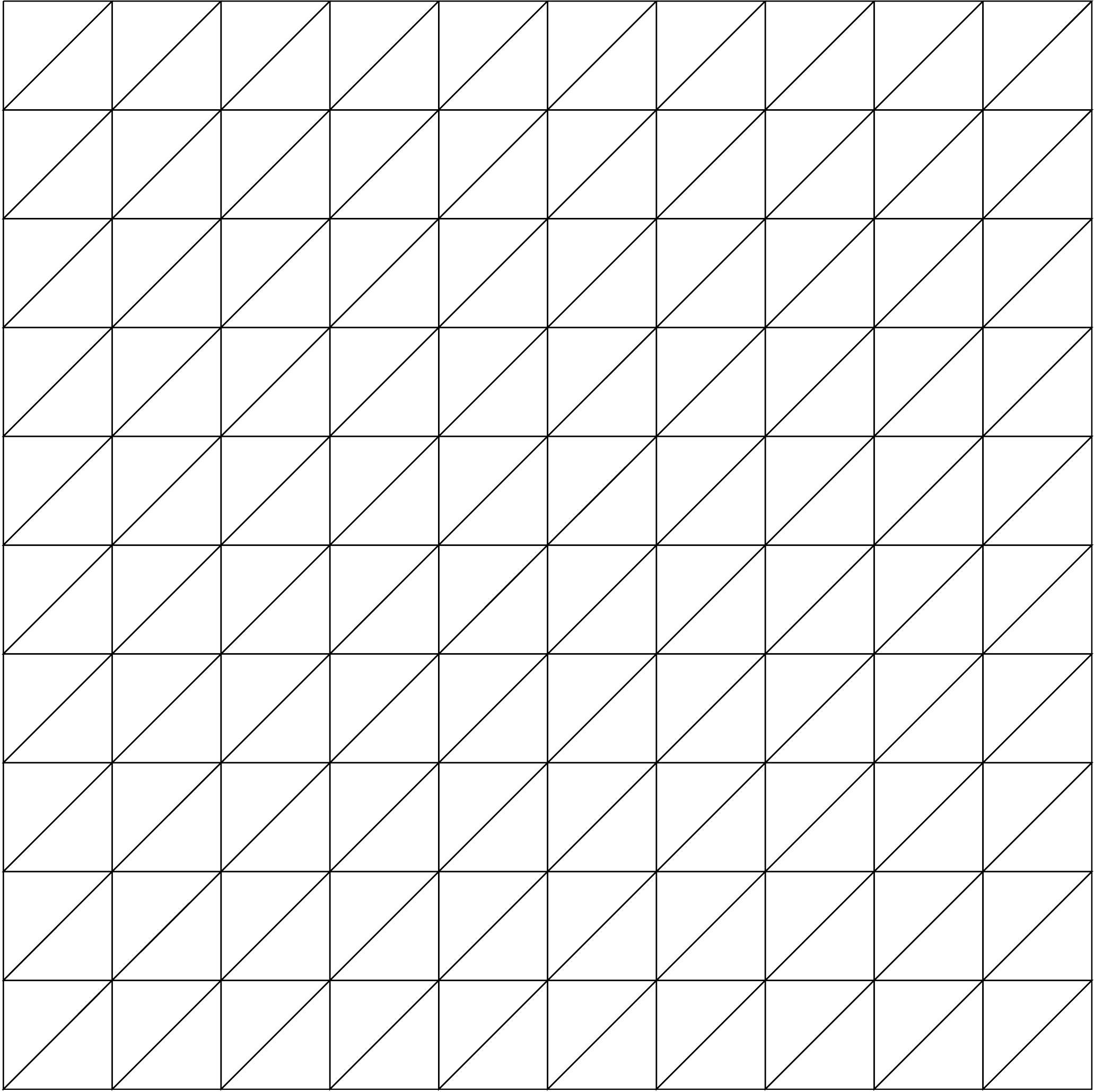}
  }
  \subfloat{
    \label{func3-mesh-first}
    \includegraphics[width=0.3\textwidth]{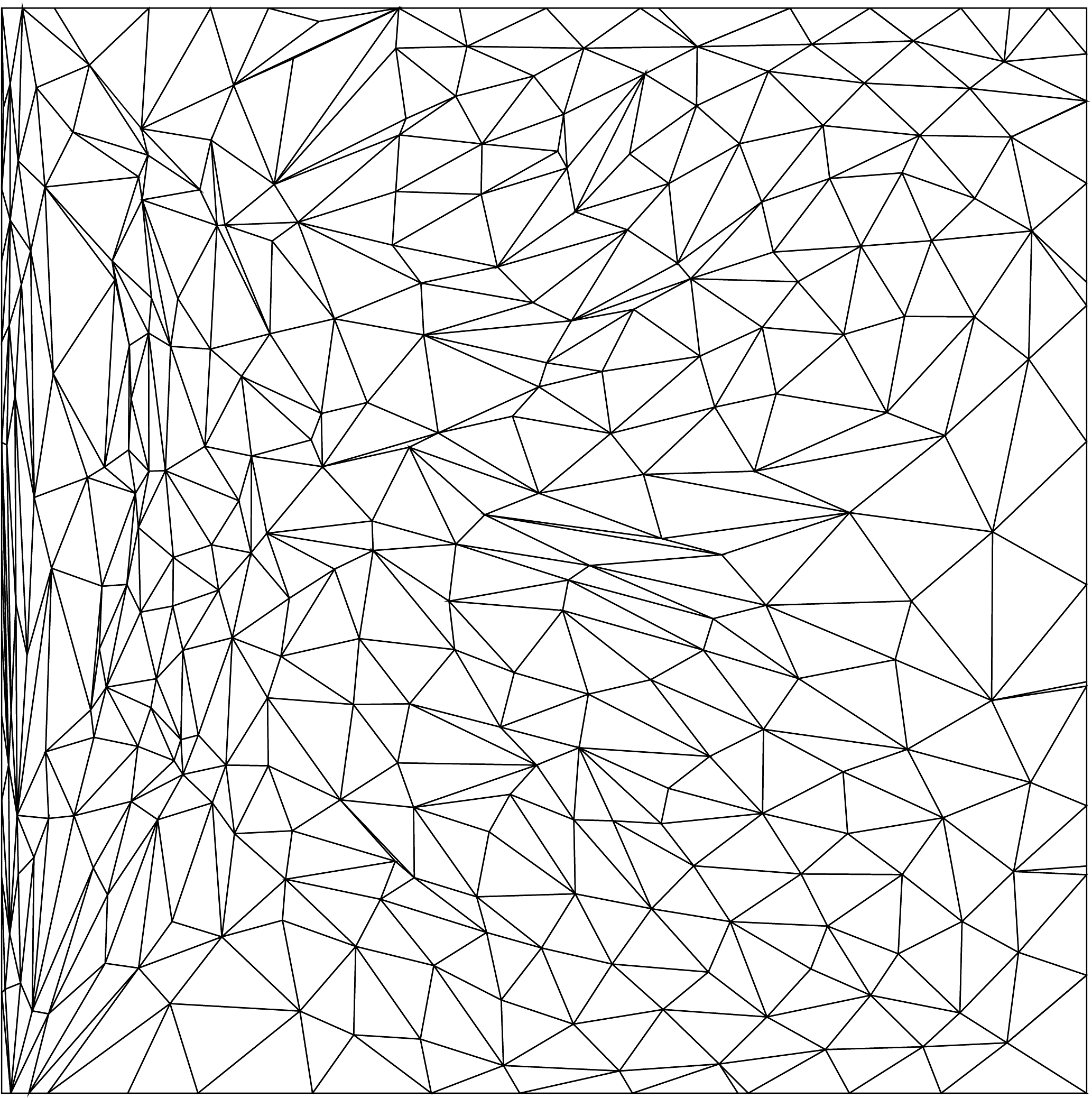}
  }
  \subfloat{
    \label{func3-mesh-tenth}
    \includegraphics[width=0.3\textwidth]{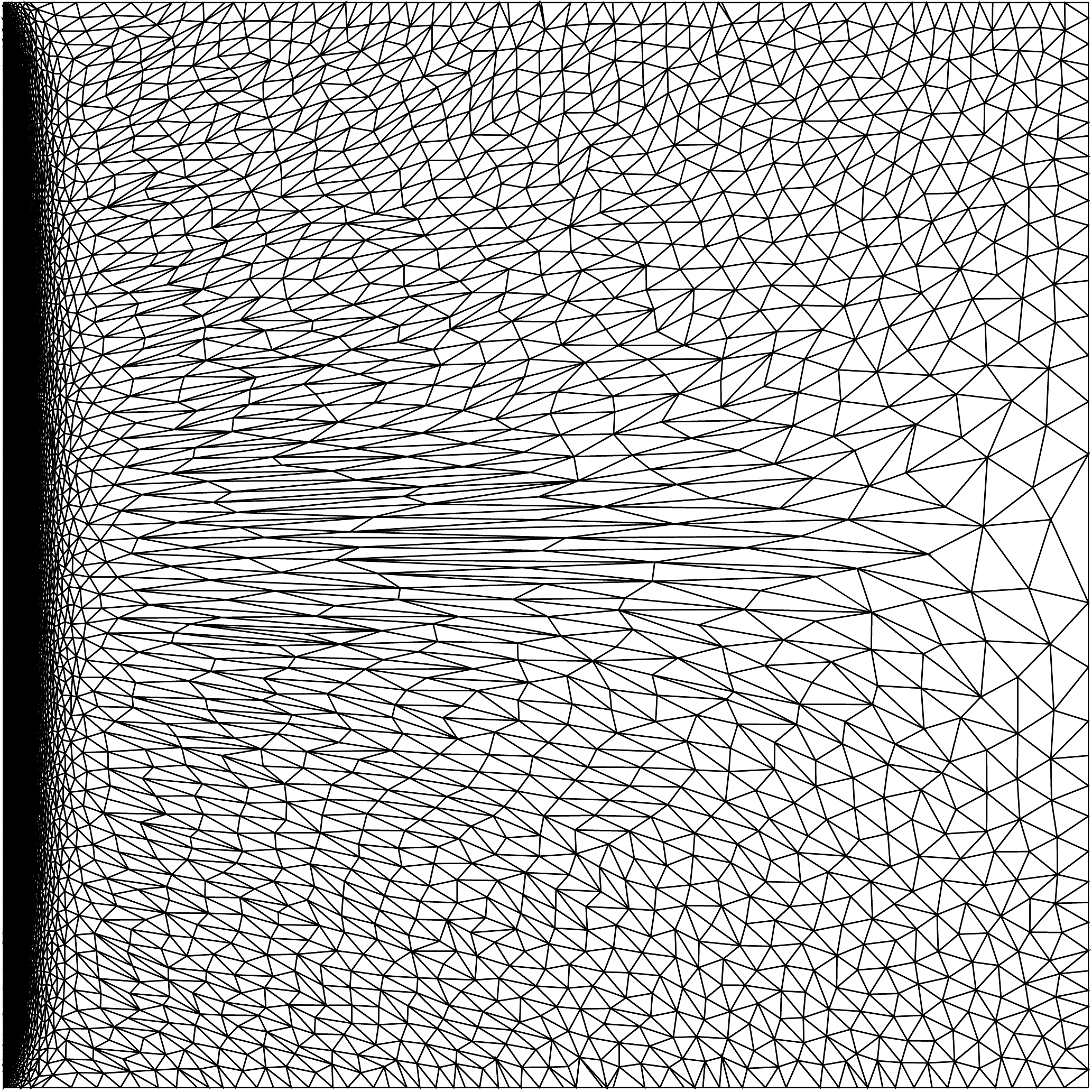}
  }
  \caption{Left: initial uniform mesh with 121 vertices. Middle: first adapted mesh with 324
    vertices. Right: final mesh with 8559 vertices.}
  \label{func3-meshes-first}
\end{figure}

The convergence of the algorithm for the element-based residual method is
assessed. Starting from a relatively coarse initial mesh, the tolerance $TOL$
is set to $0.125$ and the loop is run for $40$ iterations (which for our
purposes will be more than sufficient). See Figure \ref{func3-meshes-first} for
initial and adapted meshes. In the context of Algorithm \ref{loop}, the edge
swapping/node movement step is always run $3$ times. Additionally, the
adaptation step \ref{itm:adap-rep} is only run once before the solution and the
estimator are recomputed. The point of view taken here is that the adaptation
algorithm should not be run too long before recomputing the solution and error
estimator. For comparison, we computed an example where step
\ref{itm:adap-rep} from Algorithm \ref{loop} is run twice, as opposed to just
once. From Figure \ref{fig:loop-rep}, repeating the loop initially calls for too
much refinement. This is likely due to a loss of accuracy of the estimator on
coarse meshes (see Figure \ref{func3-errdist} and the related discussion).

\begin{figure}[ht]
  \centering 
  \subfloat[Solid line: applying the adaptation step from Algorithm
  \ref{loop} twice before recalculating the solution. Dotted line: applying the
  adaptation step once.]{
    \label{fig:loop-rep}
    \includegraphics[width=0.47\textwidth]{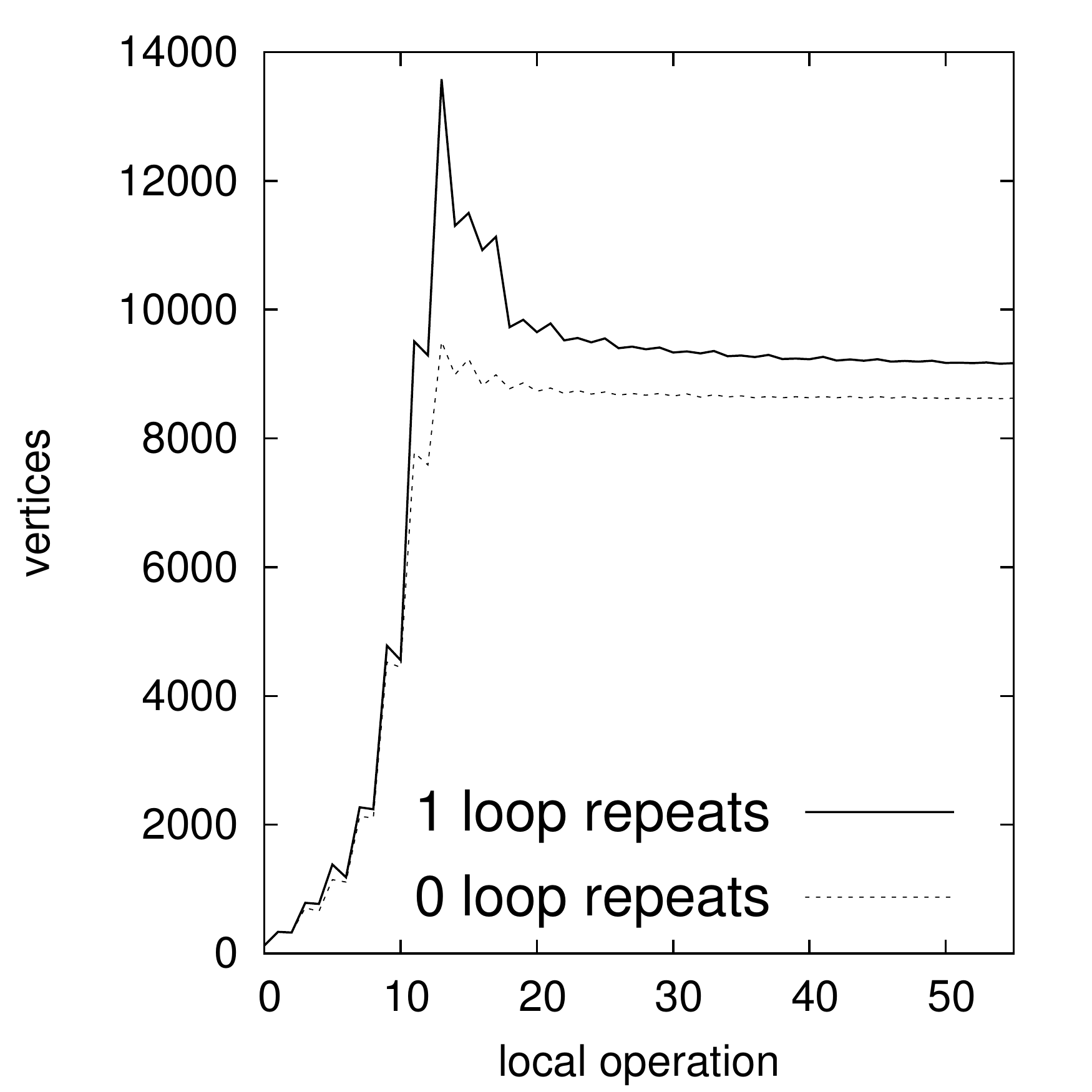}
  }
  ~~
  \subfloat[Solid line: maximum displacement. Dotted line: average
  displacement.]{
    \label{fig-nodes-distance}
    \includegraphics[width=0.47\textwidth]{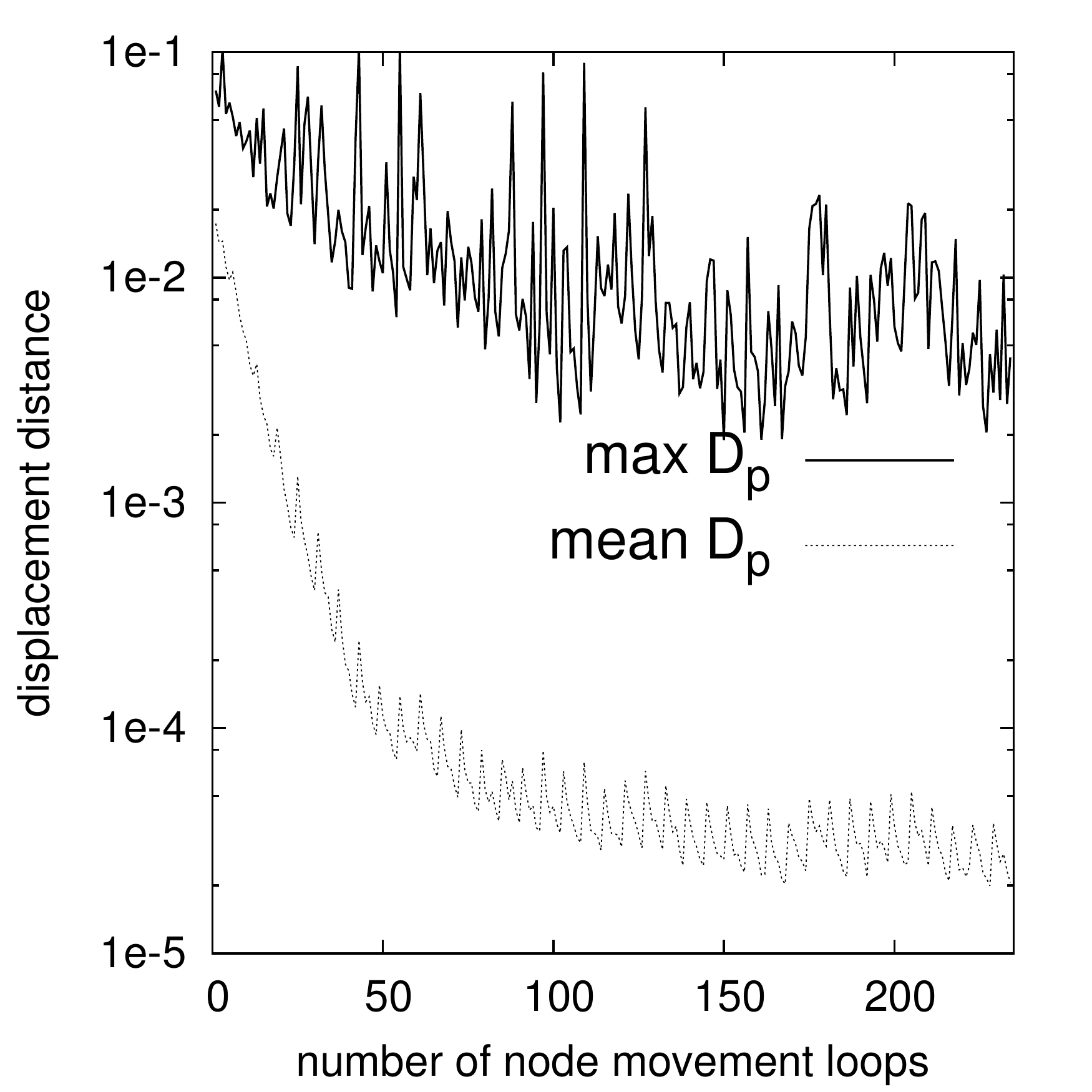}
  }

  \caption{Left: Number of vertices after refinement/derefinement. 
    Right: Maximum and average displacement of nodes for cumulative node
    movement loops.}
\end{figure}

Table \ref{func3-operations-complete} records the number of refinements,
derefinements and edge swappings performed at each iteration. The columns for
edge swapping include the sum of three separate edge swapping loops. In all
cases, note that the number of operations performed becomes small by about
$10$ iterations.

\begin{table}[h]
  \footnotesize
  \centering
  \begin{tabular}{r|rr|rr|rr|rr}
    it.&\multicolumn{2}{|c}{refinement}&\multicolumn{2}{|c}{derefinement}&\multicolumn{4}{|c}{swapping}\\
             &      &                        &      &                          &\multicolumn{2}{|c}{after refinement}&\multicolumn{2}{|c}{after derefinement}\\
             &edges &\%                      &nodes &\%                        &edges &\% &edges &\% \\
    \hline
    1 & 212  & 175.21 & 9 & 2.70 & 500 & 53.48 & 251 & 27.61\\
    2 & 415  & 128.09 & 47 & 6.36 & 1049 & 49.23 & 418 & 20.93\\
    3 & 487  & 70.38  & 55 & 4.66 & 1179 & 34.40 & 569 & 17.41\\
    4 & 996  & 88.61  & 20 & 0.94 & 1616 & 25.96 & 768 & 12.44\\
    5 & 2383 & 113.48 & 78 & 1.74 & 3275 & 24.66 & 1235 & 9.45\\
    6 & 3295 & 74.80  & 256 & 3.32 & 4556 & 19.91 & 1675 & 7.57\\
    7 & 2031 & 27.28  & 507 & 5.35 & 3926 & 13.95 & 1380 & 5.18\\
    8 & 265  & 2.95   & 497 & 5.38 & 1902 & 6.94 & 990 & 3.82\\
    9 & 189  & 2.16   & 223 & 2.50 & 1173 & 4.43 & 677 & 2.62\\
    10 & 87  & 1.00   & 132 & 1.50 & 777 & 2.98 & 535 & 2.08\\
    \vdots&\vdots&\vdots&\vdots&\vdots&\vdots&\vdots&\vdots&\vdots\\
    40 & 5   & 0.06   & 4 & 0.05 & 153 & 0.60 & 110 & 0.43\\
    \end{tabular}   
  \caption{Number of local operations for complete adaptation loop.}
  \label{func3-operations-complete}
\end{table} 

Convergence of node movement is measured by norm of the displacement of
individual nodes. For a node $p$ denote by $D_p$ the norm of its
displacement. Figure \ref{fig-nodes-distance} plots the max and mean displacement
for each iteration of node movement during the adaptation loop. While the
maximum displacement remains of the order $10^{-2}$, this value represents only
a few outlier cases, with the average displacement occurring between $10^{-5}$
and $10^{-4}.$




\paragraph{Control of the energy norm of the error.}
\label{subsubsec:first-err-energy}

\begin{figure}[h]
  \centering
  \subfloat{
    \includegraphics[width=0.47\textwidth]{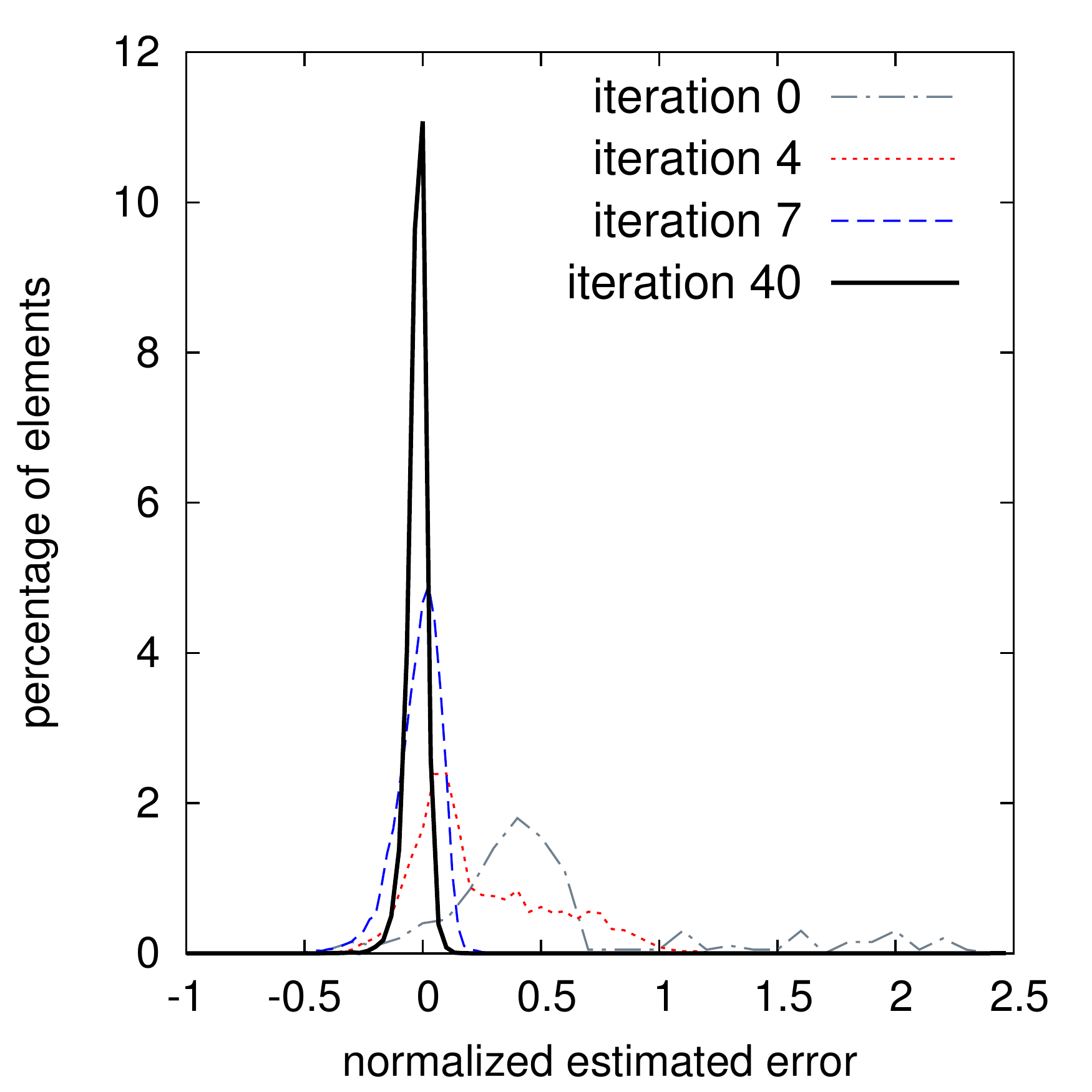}
  }
  \subfloat{
    \includegraphics[width=0.47\textwidth]{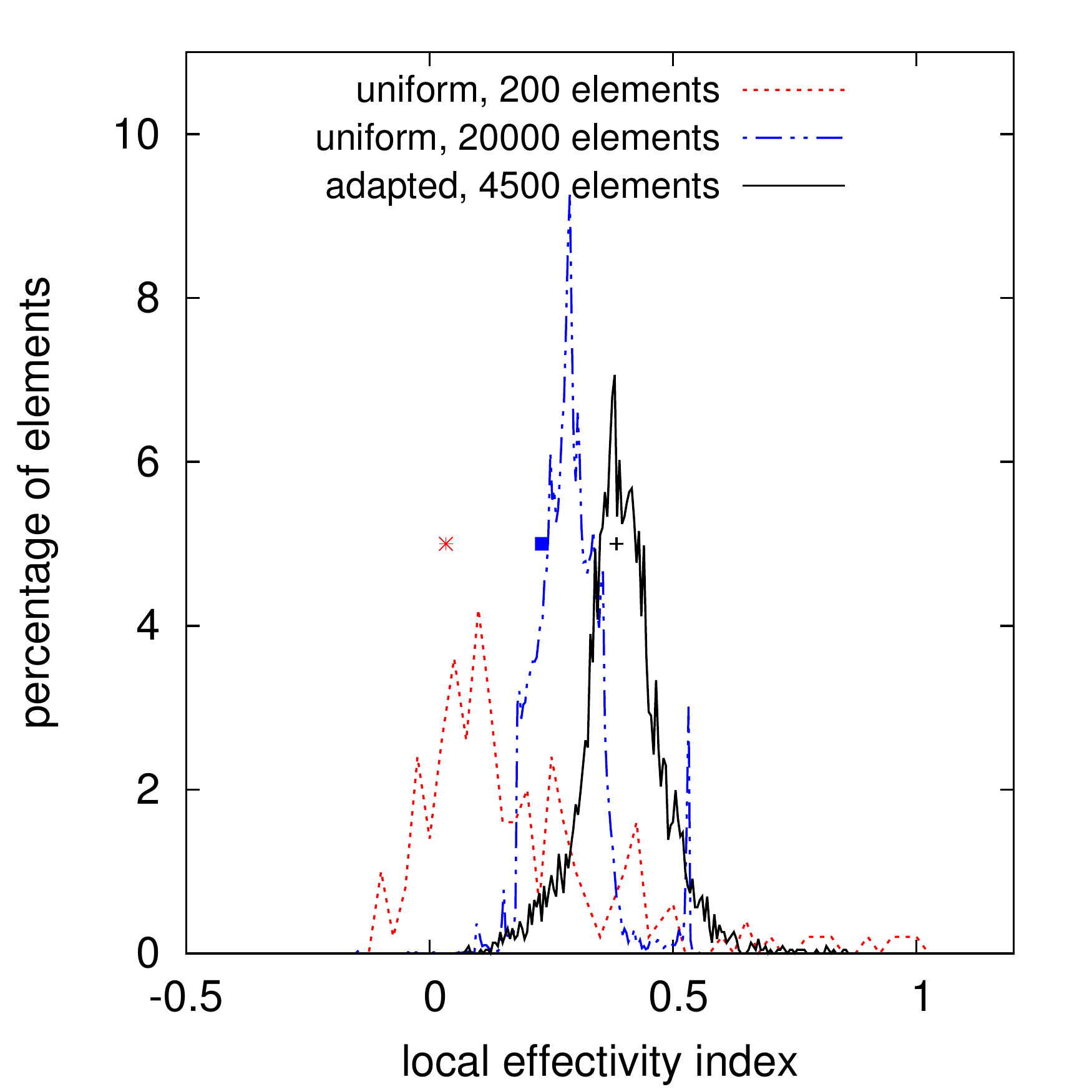}
  }
  \caption{Left: distribution of the error over elements. Middle: standard
    deviation of error distribution. Right: Distribution of the local
    effectivity. The global effectivity index is
    indicated by a red star for the coarse uniform mesh, by a blue square for
    the fine uniform mesh, and a black plus for the adapted mesh.}
  \label{func3-errdist}
\end{figure}

Next we assess the performance of the algorithm towards equidistributing the
error. The distribution of the estimated error for different iterations is
plotted in Figure \ref{func3-errdist}. The error is normalized by taking
$e_{K}=\log_{10}\left(\frac{\eta_K}{TOL/\sqrt{N_T}}\right)$. From the
figure, we see that after successive iterations the error increasingly tends to
cluster towards the target error and the distribution tends to be more normal.
Furthermore, the standard deviation of the error decreases from $0.58$ on the
initial mesh to $0.041$ on the final mesh.

Next we establish numerically the equivalence between the exact and estimated
error. The motivation is to assess to what degree we can expect to control the
exact error, both locally and globally, by equidistributing the estimated error
as in Figure \ref{func3-errdist}. Define the global effectivity index with
respect to the energy norm by $ei=\frac{\eta}{\vertiii{e_h}},$ where
$\eta=\sqrt{\sum_K\eta_K^2}.$ Theorem \ref{lem-est-aniso} says that globally the
energy norm of the error is equivalent to the estimated error, so that the
effectivity index assesses this equivalence on a given mesh. Ideally, the
effectivity index satisfies $ei\rightarrow 1$ as the mesh element size goes to
$0$, in which case we say that the estimator is asymptotically exact. This
effectivity index is studied for instance in \cite{pic03a} and \cite{micper06},
where it was observed to remain reasonably low for adapted meshes (between 2 to
5). Furthermore, for meshes adapted with target error $TOL$, the index remains
bounded as $TOL\rightarrow 0$, see for instance \cite[Table 3.7]{pic03a}.  Thus,
while the estimator is not asymptotically exact, it is clearly equivalent to the
exact error.

We define the local effectivity index for a triangle $K$ by
\begin{equation}\label{eff-loc}
  ei_K=\frac{\eta_K}{|e_h|_{1,K}}.
\end{equation}
(Recall that since $A=I$ in (\ref{eq-lap}) the energy norm is just the $H^1$
seminorm.)  The quantity (\ref{eff-loc}) measures the equivalence of the exact
and estimated error at the level of the element. In Figure \ref{func3-errdist}
we plotted the distribution of (\ref{eff-loc}) for a few meshes. For the coarse
uniform mesh with $200$ elements, note that while the global effectivity index
is quite low ($ei=1.08$), the distribution of the local effectivity index is
spread out, with a large upper tail. On the other hand, the finer uniform mesh
with $20000$ elements has a higher global effectivity index ($ei=1.70$) with a
smaller tail, suggesting that the accuracy of the estimator improves with
refinement. We also show the distribution for a relatively coarse adapted mesh
with $4500$ elements. While the global effectivity index is higher ($ei=2.42$),
the local effectivity index is more closely distributed about the global
effectivity index. What appears to happen is that refinement exaggerates the
overestimation of the error that already occurs in uniform meshes.





\paragraph{Control of the $L^2$ norm of the error.}
\label{subsubsec:first-err-l2}

In the remainder of the subsection, we will assess numerically the lower and
upper bounds for the $L^2$ error given in (\ref{eta-scale-equiv}), and briefly
present some results using the estimator (\ref{eqn:eta-scaled}) for mesh
adaptation.


\begin{figure}[ht]
  \centering 
  \subfloat[Estimated error (\ref{eq-est-aniso}) and (\ref{eqn:eta-scaled})
  vs. the exact $L^2$ error for a mesh adapted using the original estimator
  (\ref{eq-est-aniso}).]{
    \label{fig:est-vs-l2-nonscale} 
    \includegraphics[width=.47\textwidth]{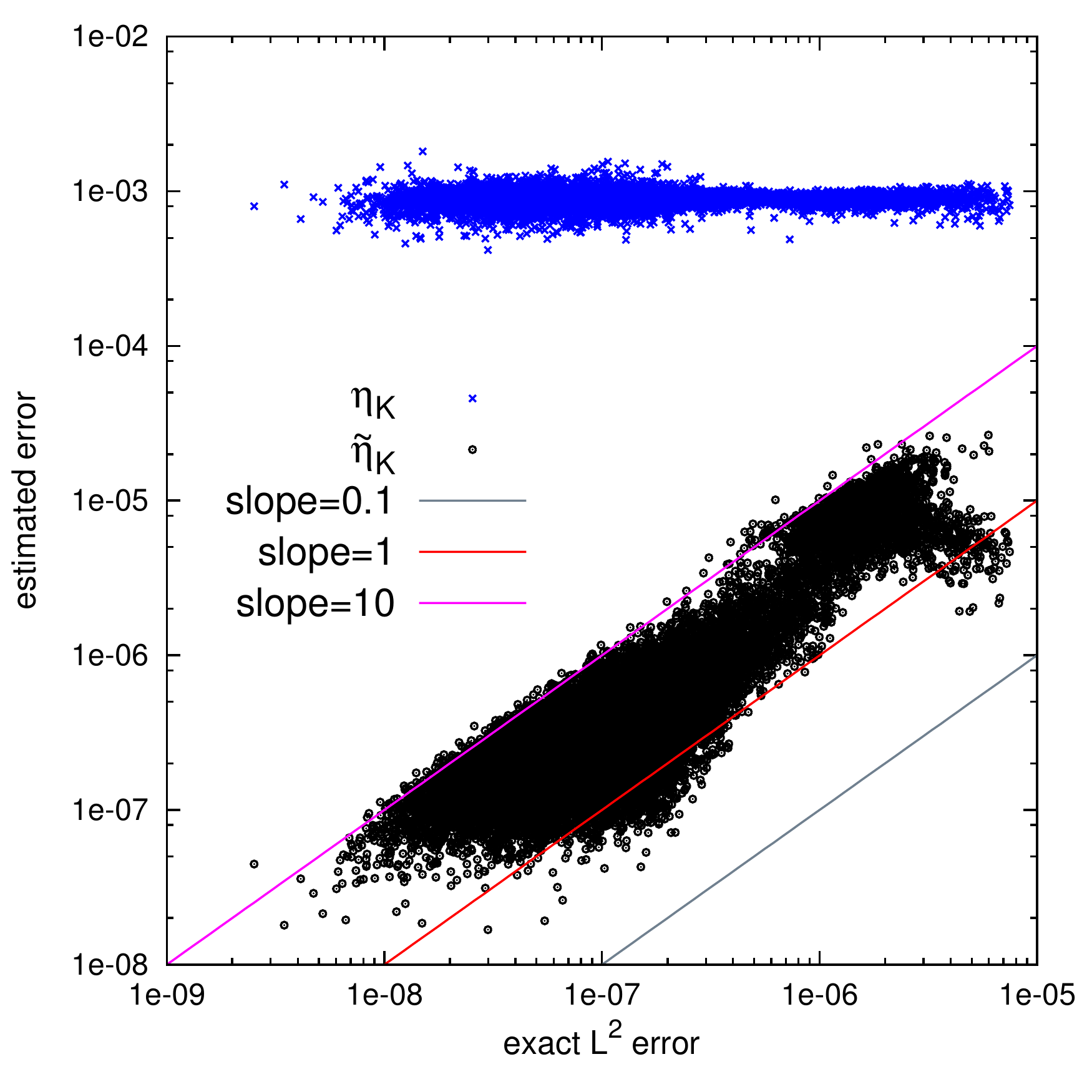}
  }
  ~~
  \subfloat[Estimated error (\ref{eqn:eta-scaled})
  vs. the exact $L^2$ error for two meshes adapted using the scaled estimator
  (\ref{eqn:eta-scaled}).]{
    \label{fig:est-vs-l2-scale} 
    \includegraphics[width=.47\textwidth]{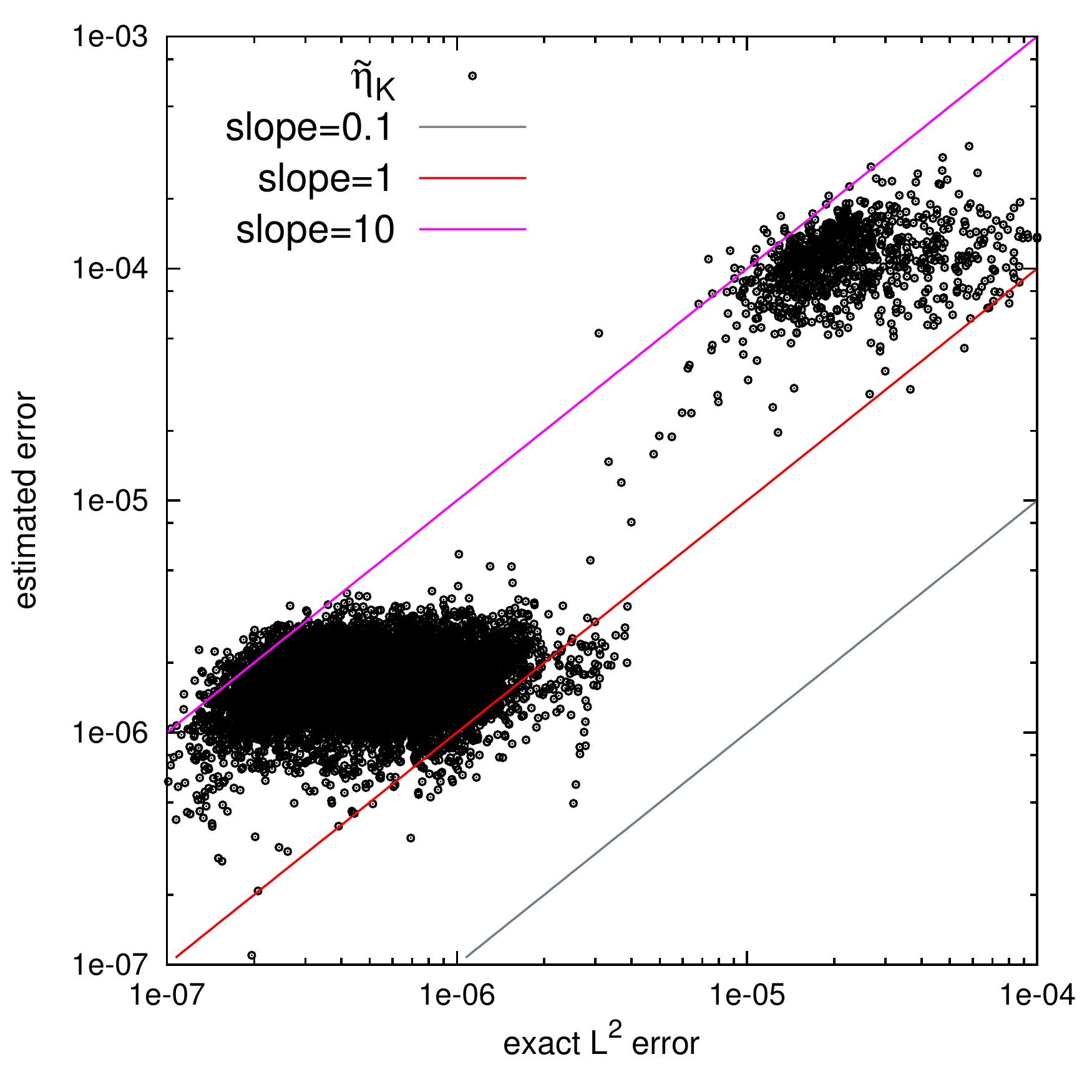}
  }
  \caption{Estimated error vs. exact $L^2$ error over elements.}
  \label{fig:est-vs-l2}
\end{figure}

Setting $TOL=0.125$, the final adapted mesh using $\eta_K$ mesh has about
$18000$ elements. Figure \ref{fig:est-vs-l2-nonscale} records for each element the
estimated error $\eta_K$ (in blue) and $\tilde{\eta}_K$ (in black) vs. the exact
$L^2$ error $\|e_h\|_{0,K}$. While $\eta_K$ remains within less than $1$ order
of magnitude, the exact $L^2$ error is spread by about $3$ orders. Therefore,
equidistributing the estimator $\eta_K$ does not lead to equidistribution of the
$L^2$ error. In verifying the lower and upper from (\ref{eta-scale-equiv}), we
see that the local effectivity index
$\tilde{ei}_K=\frac{\tilde{\eta}_K}{\|e_h\|_{0,K}}$ remains between about $0.1$
and $10,$ with the lower bound appearing sharp.

Next we adapt the mesh using the scaled error $\tilde{\eta}_K$. The mesh is
adapted using the hybrid error approach from \cite{boiforfor12} for the
hierarchical estimator. That is, edge refinement and node removal are used to
control the global $L^2$ error level (here using (\ref{eqn:eta-scaled})), while
edge swapping and node movement are used to equidistribute the error by
minimizing the energy norm (here using (\ref{eq-est-aniso})). Results of two
meshes adapted using different target error levels are presented in
Figure \ref{fig:est-vs-l2-scale}: a mesh with about $1000$ elements (top right)
and one with $14000$ (bottom left). The spread of the $L^2$ error is
significantly lower, going from $3$ orders of magnitude to about $1.5$.




\begin{figure}[h]
  \centering
  \subfloat{
    \label{func3-energy-scale-v-non}
    \includegraphics[width=0.50\textwidth]{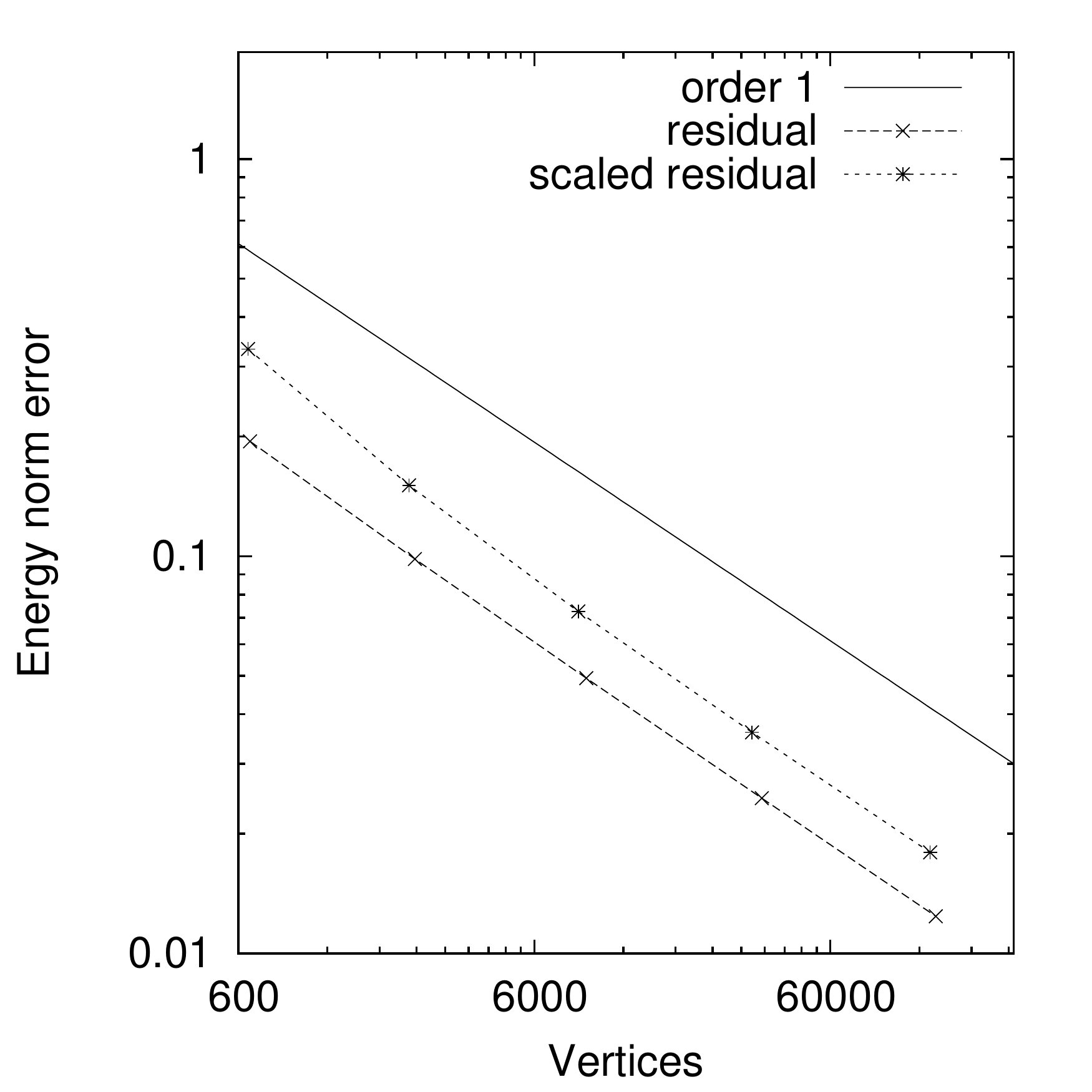}
  }
  \subfloat{
    \label{func3-L2-scale-v-non}
    \includegraphics[width=0.50\textwidth]{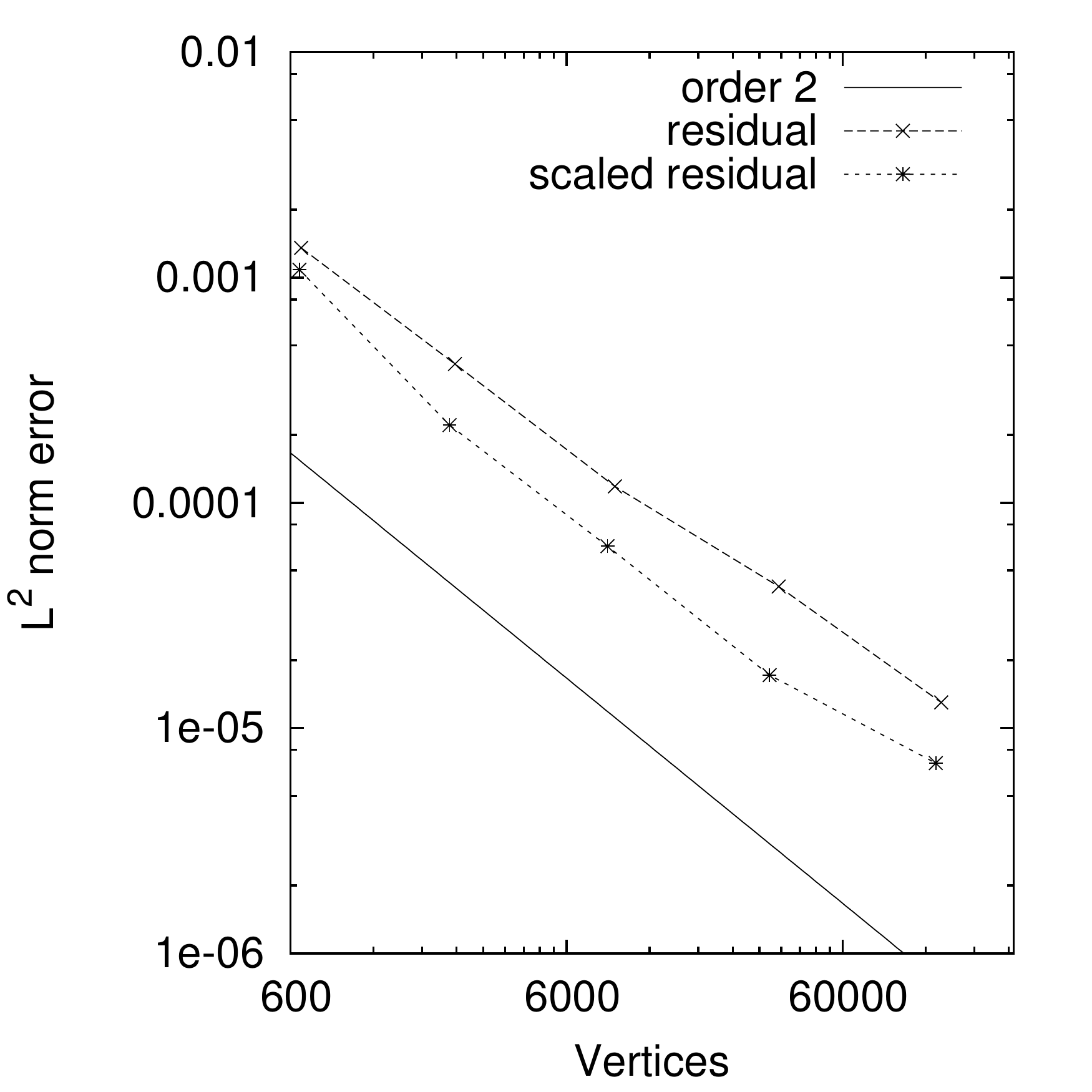}
  }
  \caption{Energy norm (left) and $L^2$ norm (right) error calculations for
    $u_1$ for residual estimators.}
  \label{func3-err-scale-v-non}
\end{figure}

We compare global error calculations using the scaled and non-scaled estimator in
Figure \ref{func3-err-scale-v-non}. Clearly, the non-scaled estimator results
in lower energy norm error for the same degrees of freedom. Moreover, as
predicted, the scaled error significantly improves the results for the $L^2$
error.


\begin{figure}[ht]
  \centering
  \subfloat{
    \label{fig:mesh-func3-orig}
    \includegraphics[width=0.48\textwidth]{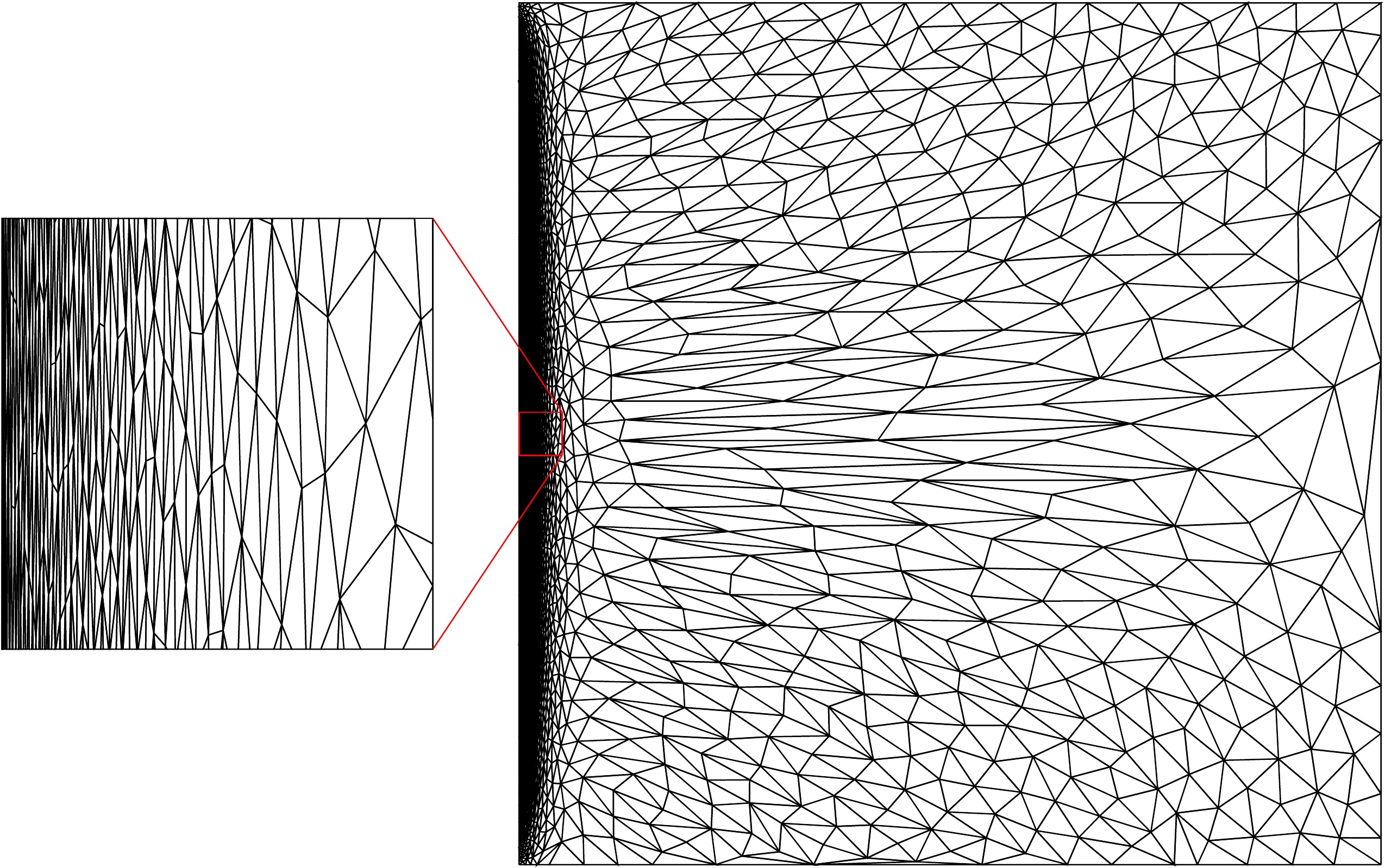}
  }
  \subfloat{
    \label{fig:mesh-func3-scaled}
    \includegraphics[width=0.48\textwidth]{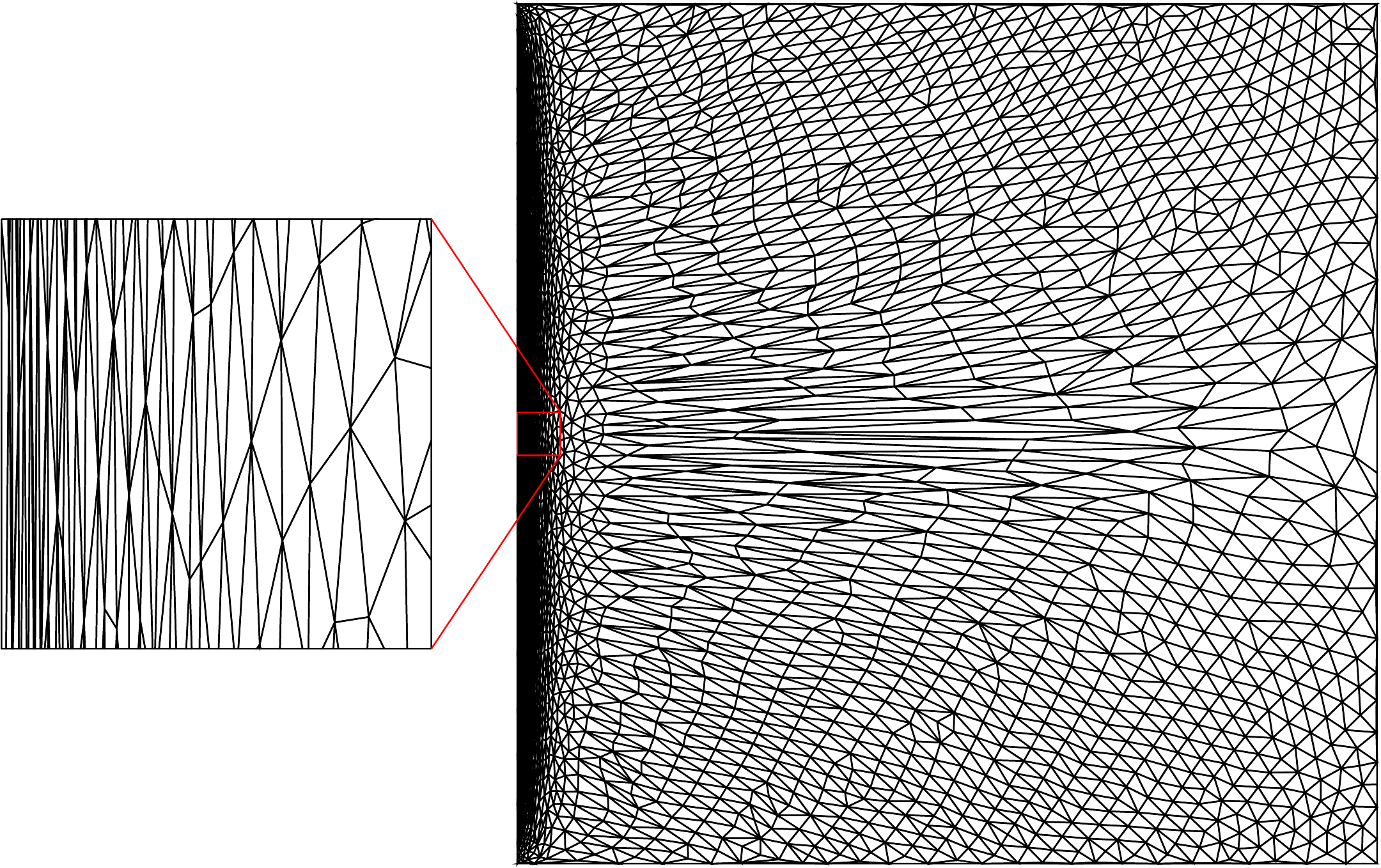}
  }
  \caption{Adapted meshes with $2500$ vertices using the $H^1$
    estimator (left) and scaled estimator (right).}
  \label{fig:meshes-func3-origscaled}
\end{figure}

Lastly, we compare meshes adapted with the scaled and unscaled estimators in
Figure \ref{fig:meshes-func3-origscaled}. Both meshes have roughly the same
number of vertices/elements, but the distribution of the elements for the mesh
adapted using the scaled estimator is much more spread throughout the domain,
while that in the original estimator tends to concentrate near the boundary
$x=0.$ This observation is really not surprising, since scaling the estimator
by the smallest eigenvalue will permit elements to be larger in areas where the
$H^1$ error is largest, such as the boundary layer. Generally, we expect that
the error converges at a higher order in the $L^2$ norm than for the $H^1$
seminorm error. But as seen in Figure \ref{fig:est-vs-l2-scale}, this higher
order is not just global (at the level of the domain), but local (element
level). This observation about mesh quality being related to the norm will be
further confirmed in what follows when we consider the hierarchical estimator,
which natively controls the $L^2$ error.

\subsubsection{Comparison of the adaptation methods}
\label{subsubsec:first-conv}


\subparagraph*{Qualitative comparison.}
\label{subsubsec:first-qual}

Figure \ref{func3-meshes} presents examples of adapted meshes with about $2500$
vertices produced by each method. In all cases, we see that the meshes contain
elements that are very stretched near the boundary layer. Note that in
general the meshes obtained from the residual estimators tend to have more
elements near the boundary layer, while the meshes from Hessian and hierarchical
methods tends to be more spread out. The difference in mesh density is likely
due to the target norm used by each method. As discussed in the previous
section, the target norm is related to the local order of convergence, which
affects local element size.

Another note is that the mesh for the Hessian is quite regular in the top and
bottom right corners. The initial mesh is regular, consisting of right
triangles as in Figure \ref{func3-mesh-init}, so what seems to be happening is
that in these regions the main operation performed is edge refinement. In
particular, node displacement appears to be less smooth for the Hessian.
Repeating the adaptation loop starting from a non-uniform mesh does in fact
result in a final mesh which is not regular.

\begin{figure}[!htbp]
  \centering
  \includegraphics[width=0.8\textwidth]{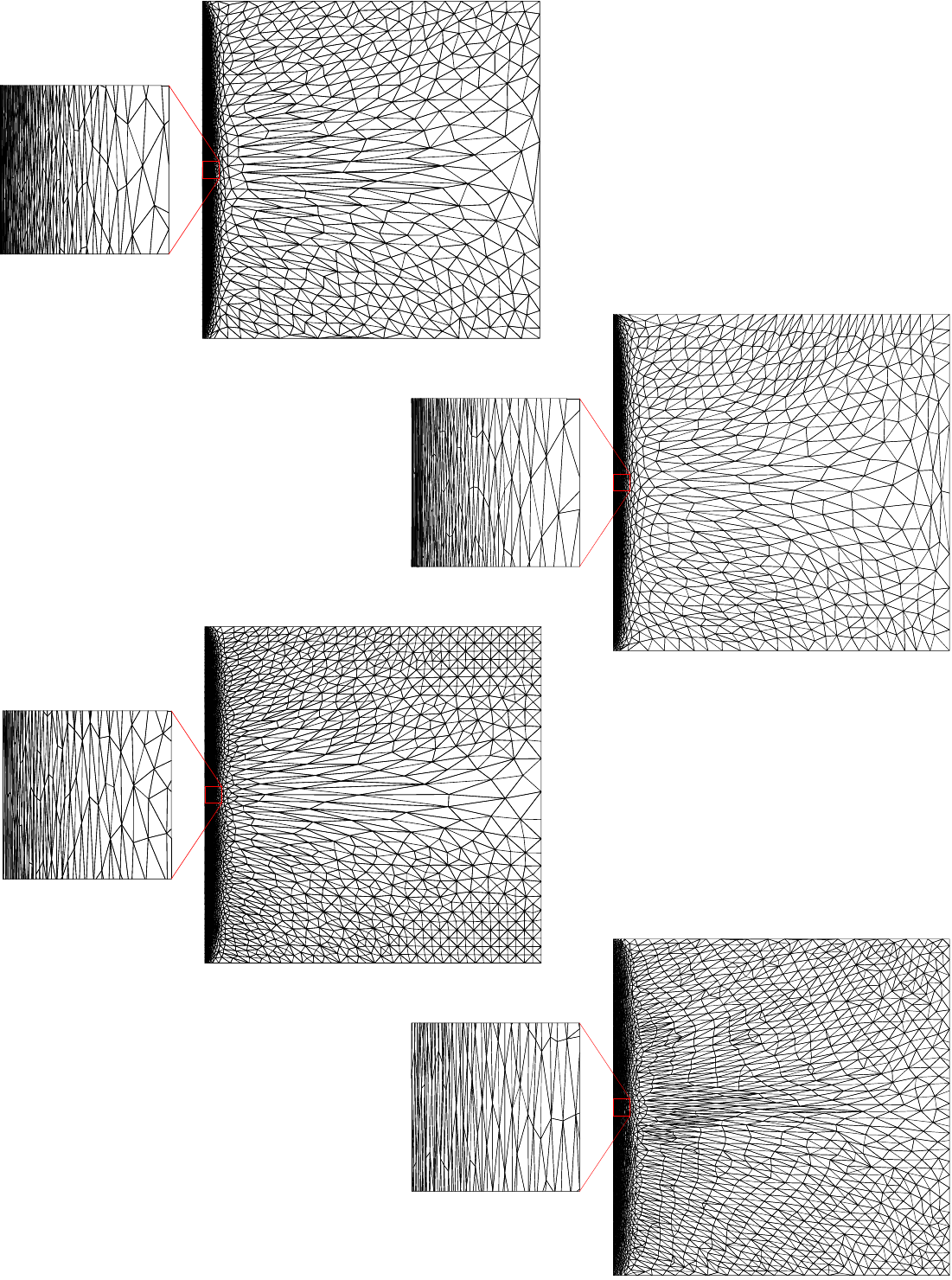}
  \caption{Adapted meshes for $u_1$ with approximately 2500 vertices,
    with zoom near the boundary at $x=0$. From top to bottom: residual
    (element), residual (metric), Hessian, hierarchical.}
  \label{func3-meshes}
\end{figure}


\paragraph{Analytical comparison.}
\label{subsubsec:first-ana}


The error is reported in Figure \ref{func3-err} as a function of the number of
vertices. Recall that for a regular mesh in $2D$, the number of vertices is
roughly proportional to $(\frac{1}{h})^2$, so that the theoretically optimal
(logarithmic) slope corresponding to (\ref{order1}) is $-1/2$, while for
(\ref{order2}) it is $-1$.

Figure \ref{func3-err} reports the error in the energy and $L^2$ norm. We see
that all methods approach the theoretical rate of convergence for the energy
norm. Moreover the hierarchical method, which reports the largest error,
remains about $1.3$ times higher than the residual element-based method, which
reports the smallest error. The convergence for the $L^2$ norm, on the other
hand, appears to be more erratic, with none of the methods achieving the optimal
rate of convergence. Here the hierarchical method reports the lowest error, the
residual methods report an error $2$ to $3$ times as large, while the Hessian
method reports an error about $4$ to $5$ times as large. Note that for both,
the energy and $L^2$ norms, the results for both residual methods are close.


\begin{figure}[h]
  \centering
  \subfloat{
    \label{func3-energy}
    \includegraphics[width=0.50\textwidth]{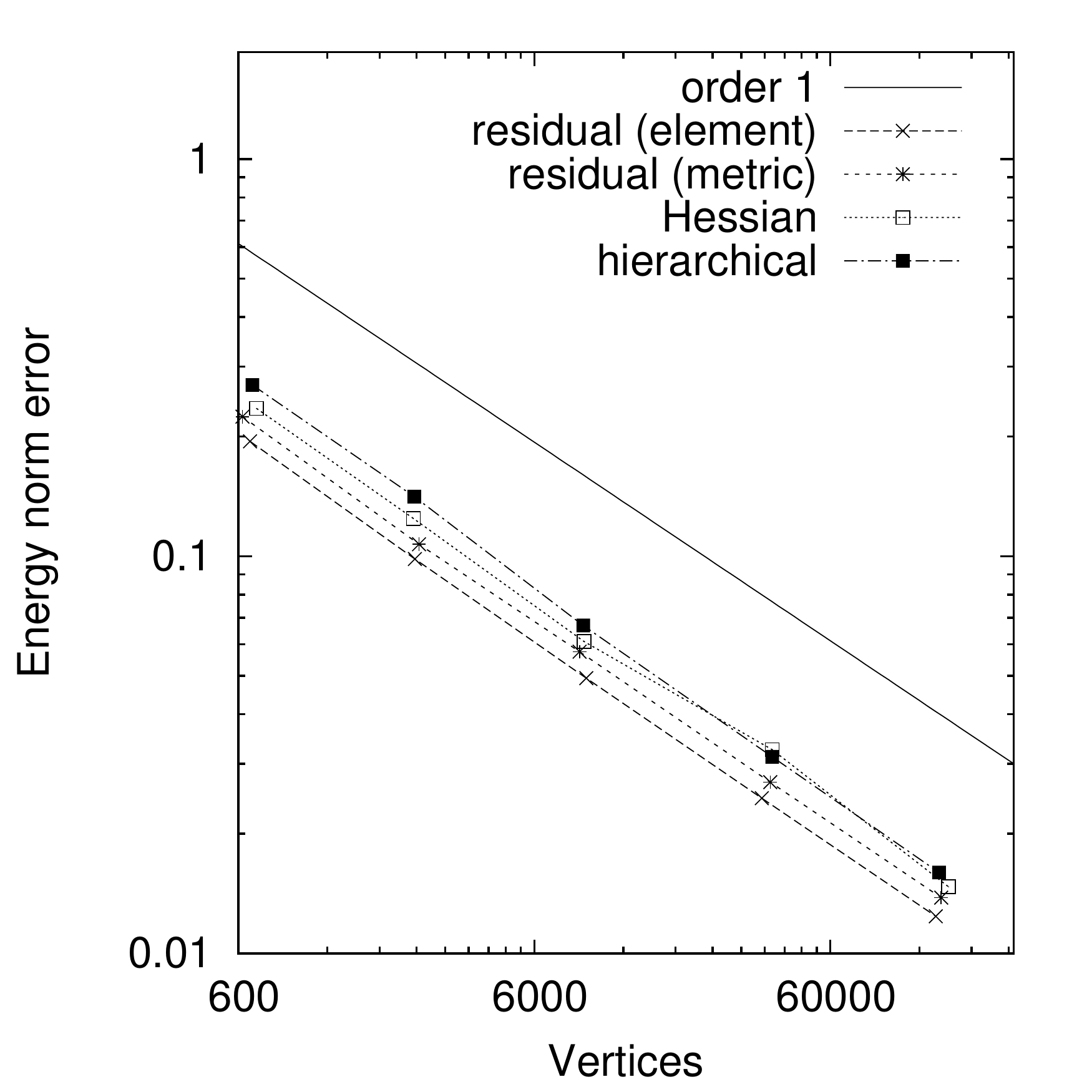}
  }
  \subfloat{
    \label{func3-L2}
    \includegraphics[width=0.50\textwidth]{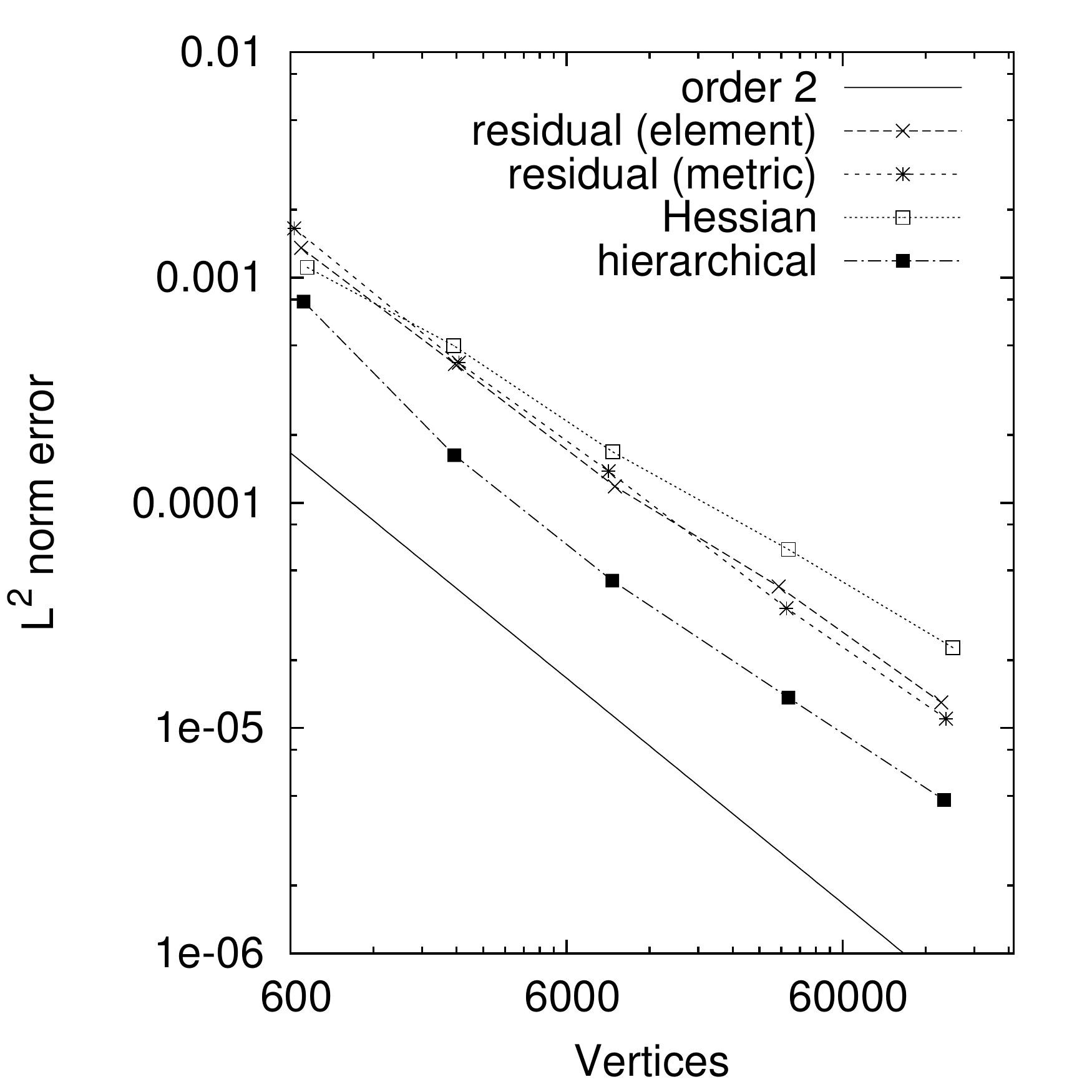}
  }
  \caption{Energy norm (left) and $L^2$ norm (right) error calculations for $u_1$.}
  \label{func3-err}
\end{figure}


In Table \ref{func3-dist-final} we record the mean and variance of the
distribution of the error over the elements. While Figure \ref{func3-err} shows
that the global energy norm is lowest for the residual methods and highest for
the hierarchical, the situation is reversed here, with the hierarchical
reporting the lowest mean error. This can be partially accounted for by the
fact that the residual methods result in the lowest standard deviation for the
energy norm of the error, which is likely the result of the equidistribution of
the estimated error achieved by the adaptive method. For the $L^2$ error, as
expected the hierarchical method reports the lowest mean and standard deviation.

We remark that the target norm for the Hessian adaptation is $L^\infty,$ so it
is possible that the Hessian does the best job equidistributing the error in the
$L^\infty$ norm. We did not calculate the $L^\infty$ norm. Furthermore, it
should be noted that in \cite{losala11b}, a Hessian-based error estimator was
developed to control the the $L^p$ error for $1\leq p<\infty.$ Adaptation is, as
before, done by constructing a metric, which turns out to be the the same as that
discussed in Section \ref{subsec:hess} with the eigenvalues appropriately scaled
for the choice of $p$, see \cite[Section 2]{losala11b}. It is reasonable to
expect that the results for the $L^2$ error could be improved using their
estimates.

\begin{table}[h]    
  \footnotesize
  \centering
  \begin{tabular}{|l|l|l|l|l|l|}
    \hline                                         
    method                 &vertices  &mean &st. dev. &mean &st. dev. \\
    & & $\|\nabla(e_h)\|_K$  &$\|\nabla(e_h)\|_K$  & $\|e_h\|_K$ & $\|e_h\|_K$\\
    \hline                                         
    residual (element)	& 657		&5.37e-03 	& 1.20e-03    &2.16e-05	&3.16e-05\\
    residual (metric)	& 639		&6.04e-03 	& 1.74e-03    &2.63e-05	&3.93e-05\\
    Hessian			& 691		&5.27e-03 	& 3.84e-03    &2.27e-05	&2.07e-05\\
    hierarchical		& 670		&5.42e-03 	& 5.30e-03    &1.84e-05	&1.21e-05\\
    \hline                                   
    residual (element)	& 2369		&1.42e-03 	& 2.82e-04    &3.25e-06	&5.13e-06\\
    residual (metric)	& 2251		&1.62e-03 	& 4.09e-04    &3.55e-06	&4.99e-06\\
    Hessian			& 2342		&1.51e-03 	& 1.08e-03    &4.40e-06	&5.98e-06\\
    hierarchical		& 2356		&1.41e-03 	& 1.54e-03    &2.14e-06	&1.08e-06\\
    \hline            
    residual (element)	& 8992		&3.63e-04 	& 7.24e-05    &4.51e-07	&7.65e-07\\
    residual (metric)	& 8701		&4.01e-04 	& 8.88e-05    &4.99e-07	&8.20e-07\\
    Hessian			& 8842		&3.85e-04 	& 2.56e-04    &6.63e-07	&1.09e-06\\
    hierarchical		& 8786		&3.55e-04 	& 3.62e-04    &2.96e-07	&1.72e-07\\
    \hline                       
    residual (element)	& 35218		&9.12e-05 	& 1.76e-05    &7.13e-08	&1.44e-07\\
    residual (metric)	& 33448		&1.02e-04 	& 2.14e-05    &8.16e-08	&1.58e-07\\
    Hessian			& 38290		&8.78e-05 	& 7.92e-05    &1.06e-07	&1.85e-07\\
    hierarchical		& 38205		&7.90e-05 	& 8.10e-05    &3.95e-08	&2.97e-08\\
    \hline
  \end{tabular}   
  \caption{Distribution of error.}
  \label{func3-dist-final}
\end{table} 


\paragraph{Computational performance.}
\label{subsubsec:first-comp}


Figure \ref{CPU-time} records the CPU time for the adaptation part of each
iteration of the loop. For each method, we chose the global error level so that
the final mesh has about $9000$ vertices. The number of vertices at each
iteration is recorded in Figure \ref{nodes-iter}, and the gap between lowest and
highest at the last step is about $6$\%. We find that metric adaptation
requires much less time than the other methods, and in the plot both metric
based methods appear superimposed. This result is not surprising. The only
that value we really need to keep track of is the metric tensor at each
vertex. The local error calculations are relatively insubstantial compared to
those required for element-based adaptation.

In addition, note that the element-based residual method takes roughly $5$ to
$6$ times that of the hierarchical estimator. For one thing, the residual
estimator requires the contribution from discontinuous functions. These
functions cannot be directly interpolated after performing local modifications,
and must be recomputed on each element/edge. Especially problematic is the
calculation of the singular value decomposition. Even for a $2\times 2$ matrix
$A$, it can be numerically disastrous to calculate the singular value
decomposition of $A$ directly by first computing $AA^T$ \cite{golvan96},
and instead it is recommended to use an iterative method. We have used the
implementation provided by DGESVD from LAPACK. Overall, it was found that this
computation takes between $13-18$\% of the total adaptation process. Another
contribution towards increased CPU time is due to the fact that the jump term
depends on more than one element. As mentioned in Section \ref{subsubsec:swap},
when performing edge swapping, the error needs to be calculated on an enlarged
patch as in Figure \ref{fig:ext-patch} in order to accurately compute the jump
term. The construction and handling of this patch introduces significant
computational overhead.

\begin{figure}[h]
  \centering
  \subfloat[CPU-time in seconds]{
    \label{CPU-time}
    \includegraphics[height=0.42\textheight]{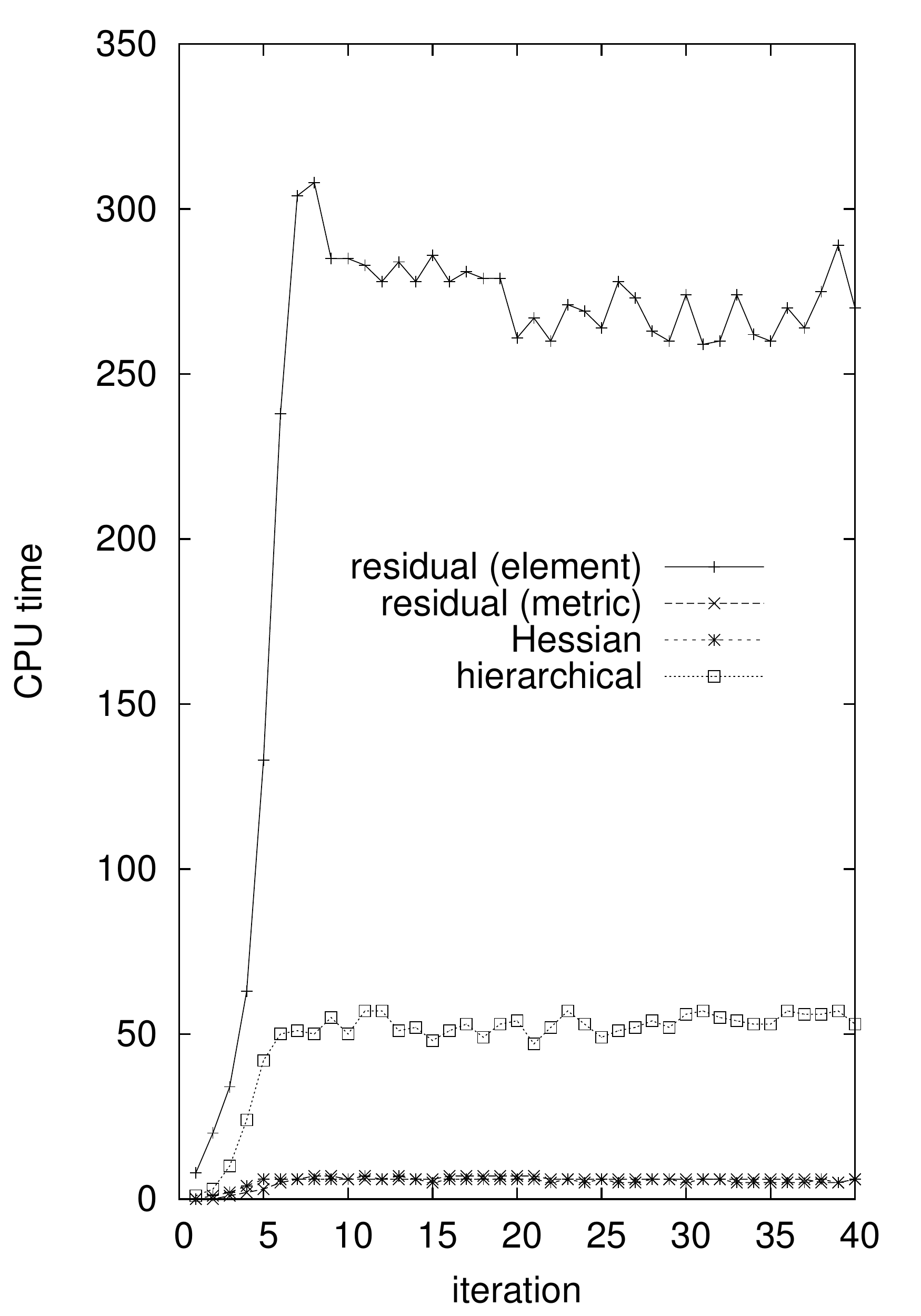}
  }
  \subfloat[Number of nodes]{
    \label{nodes-iter}
    \includegraphics[height=0.42\textheight]{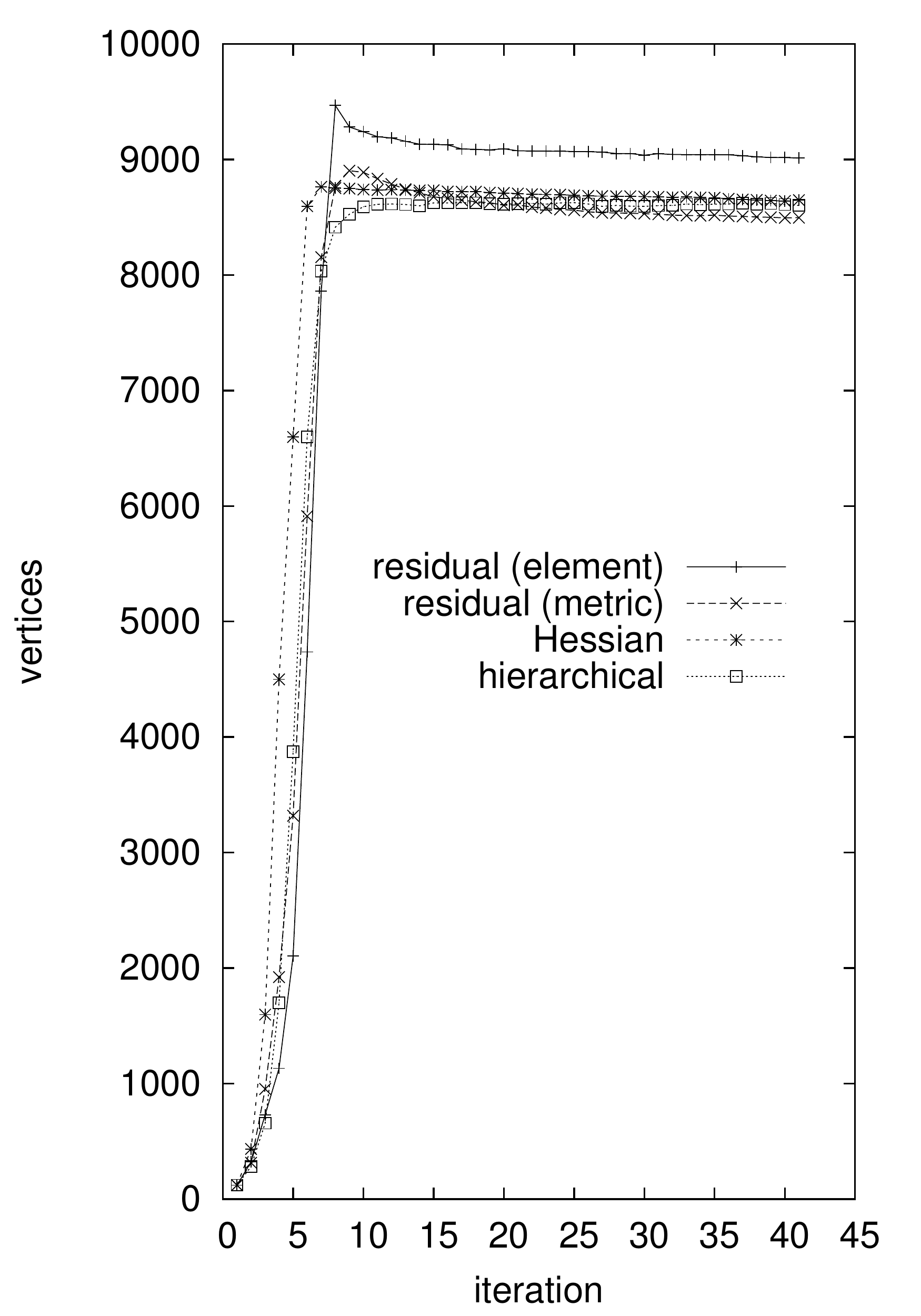}
  }  
  \caption{Left: CPU time for each iteration. Right: number of nodes at each iteration.}
  \label{cpu-nodes}
\end{figure}


\subsection{Second test case}
\label{subsec:second}


With the same parameters for problem (\ref{eq-lap}) as test case 1, we consider
the function taken from \cite{mit10}
\begin{equation*}
  u_2=\tan^{-1}(\alpha(r-r_0)),
\end{equation*}
where $r=\sqrt{(x+0.05)^2+(y+0.05)^2},$ and $r_0=0.7$. Thus, we have a circular
a wave-front type solution, centered at $(-0.05,-0.05)$ with a transition region
with thickness of order $\alpha^{-1}.$ We will run simulations with both
$\alpha=100$ and $\alpha=1000.$


\subsubsection{Qualitative comparison}
\label{subsubsec:second-qual}

In Figures \ref{wave2-meshes} we show some examples of adapted meshes. In each
case, the mesh follows what we would expect from the solution. The elements are
mainly concentrated near the wave-front where the gradient is steep in the
direction orthogonal to the wave, and with the alignment of the elements in this
region reflecting the curvature. Outside this region, variation in the solution
is reduced significantly, so that the elements can be much larger. What is
striking, however, is the difference between the mesh produced by the
hierarchical method compared to the others. For the hierarchical method, the
mesh is more spread out and less concentrated near the wave-front. As discussed
in Section \ref{subsec:first}, we attribute this difference to the target norm
used. Another feature of interest, seen in the zoom to the wave-front, is a
sub-layer of elements where the mesh is coarser. In this region, the function
is almost linear in the direction orthogonal to the wave-front, so that the
error is somewhat smaller than in the immediate surroundings.

\begin{figure}[!htbp]
  \centering
  \includegraphics[width=0.8\textwidth]{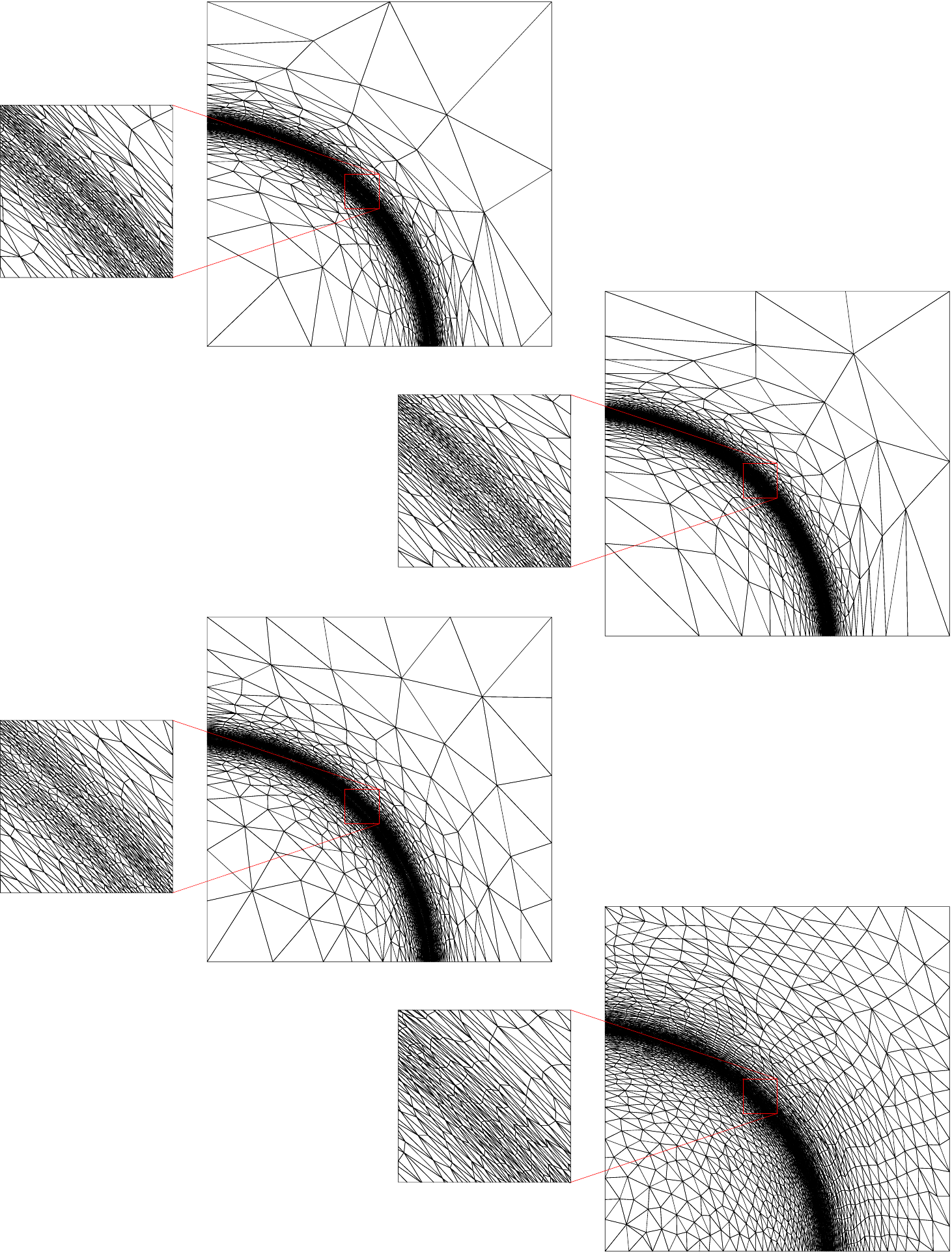}
  \caption{Adapted meshes for $u_2$, $\alpha=100$ with approximately
    3000 vertices, with zoom to wave front. From top to bottom:
    residual (element), residual (metric), Hessian, hierarchical.}
  \label{wave2-meshes}
\end{figure}



\subsubsection{Analytical comparison}
\label{subsubsec:second-ana}


\paragraph{Calculation of the residual term.}
\label{par:sterm}

This test case highlights one of the drawbacks of residual estimators. To
calculate the error $\eta_K$, we need to evaluate the integrals
$\int_K f^2\,\,\mathrm{d}x$ of the source term $f.$  (Since $A=I$ and
$u_h$ is piecewise linear, we get $R_K(u_h)=f$.)  At the wave front with
$\alpha=1000$ this value is very difficult to compute accurately. The immediate
effect on adaptation was that some elements were not being refined despite being
flagged as having large error. In particular, the algorithm reported
\begin{align}\label{res:explode}
  \int_{K_1} f^2\,\,\mathrm{d}x
  + \int_{K_2} f^2\,\,\mathrm{d}x
  & \gg \int_K f^2\,\,\mathrm{d}x,
\end{align}
where $K_1,\,K_2$ are obtained by refining an edge of $K.$

We improve the accuracy of the integral by subdivision. For a given quadrature
rule $\mathcal{Q}_K$ on an element, we divide the triangle into $4$ by splitting
the edges in half, then define the subdivided quadrature rule
$\mathcal{Q}_{K,1}$ to be that composed of four copies of the original, each
weighted $\frac{1}{4}|K|.$ The effect is that if the rule we had before was
$Ch^k$ accurate, the subdivided scheme is $\frac{C}{2^k}h^k$ accurate. We
therefore increase the accuracy without introducing a large constant from a
higher-order method. See Algorithm \ref{alg:sub-res} for implementation details.

\begin{algorithm}[h]
  \caption{Calculation of the residual on element $K$.}
  \label{alg:sub-res}
  \begin{enumerate}
  \item Calculate the residual $\tilde{R}_0$ with the original quadrature
    $\mathcal{Q}_{K,0}$.
  \item Choose $\epsilon>0$ to be small. For $i=0,1,2$ do the following:
    \begin{enumerate}
    \item \label{itm:subdiv} Subdivide the current quadrature $\mathcal{Q}_{K,i}$
      into $\mathcal{Q}_{K,i+1}$ and compute the residual $\tilde{R}_{i+1}.$
    \item If $i=2$ or if $\frac{|\tilde{R}_i-\tilde{R}_{i+1}|}{\tilde{R}_{i+1}}\leq\epsilon$
      then accept $\tilde{R}_{i+1}$ as the residual and exit.
    \end{enumerate}
  \end{enumerate}
\end{algorithm}

Since each time we subdivide, we multiply the number of Gauss points by $4$,
subdivision can quickly become expensive. We always subdivide at least once, so
that at the very least we need to compute values at $(1+4)B_G$ points, where
$B_G$ is the base number of Gauss points. Therefore, higher-order quadrature
rules are virtually unusable for subdivision, and the total number of
subdivisions never exceeds $3$. Fortunately, in our case it was sufficient to
use the single point (barycenter) integration scheme. Even still, this comes at
the high cost of $64$ Gauss points for the third subdivision. The percentage of
subdivisions that occur for an adaptation loop with $\epsilon$ from Algorithm
\ref{alg:sub-res} set to $0.05$ are reported Figure \ref{tab:subdiv}. By the
tenth iteration, additional subdivision is not significant.

\begin{table}[h]
  \footnotesize
  \centering
  \begin{tabular}{|r|l|l|}
    \hline
    it. & 2 sub. \%& 3 sub. \%\\
    \hline
    1   &  7.55    & 4.02     \\
    5   & 10.69    & 2.69     \\
    10  &  1.79    & 0.88     \\
    20  &  1.49    & 0.77     \\
    \hline
    \end{tabular}   
    \caption{Percentage of elements where additional subdivision occurs at each
      iteration of the global adaptation step.}
  \label{tab:subdiv}
\end{table} 

We remark that subdivision integration is not necessary if adapting using a
metric. There, the residual is calculated only once to compute the metric so
that issues such as (\ref{res:explode}) will not be seen during adaptation.
Furthermore, when computing the metric, the residual term is often left out
altogether to save computational time, as is done in
\cite{burpic03}. Theoretically, this simplification can be justified for the
Laplace equation as proven in \cite{kunver00}. In the case of element-based
adaptation, we found that including the residual term was necessary, since
experiments with removing the residual term generally resulted in meshes of poor
quality.



\paragraph{Global error comparison.}
\label{par:error-glob}

In Figures \ref{wave-both-energy} and \ref{wave-both-l2}, we record the error
convergence for $u_2$ for $\alpha=100$ and $\alpha=1000$. The results are very
similar to that for $u_1$: for the energy norm, the results are close, with the
element-based residual method reporting the lowest error, while for the $L^2$
error, as in Section \ref{subsec:first}, the hierarchical method reports the
lowest. The results for the $L^2$ error for $\alpha=1000$ will be discussed in
some detail here, for they clearly highlight the issue of controlling the $L^2$
norm with an estimator for the $H^1$ seminorm. We found that the $L^2$ error
oscillates over consecutive iterations when adapting with the residual in some
situations. To illustrate this issue, we reported the results in a different way
in Figure \ref{wave-both-l2}. For each target error, after the number of local
modifications and vertices has stabilized, we take the smallest and largest
error after $10$ further adaptation iterations, giving an upper and lower
envelope. With the exception of the residual element-based, where we see a
persistent spread of about $5$ to $10$\%, the envelope becomes narrow as the
number of nodes increases.

\begin{figure}[h]
  \centering
  \subfloat{
    \label{wave2-energy}
    \includegraphics[width=0.50\textwidth]{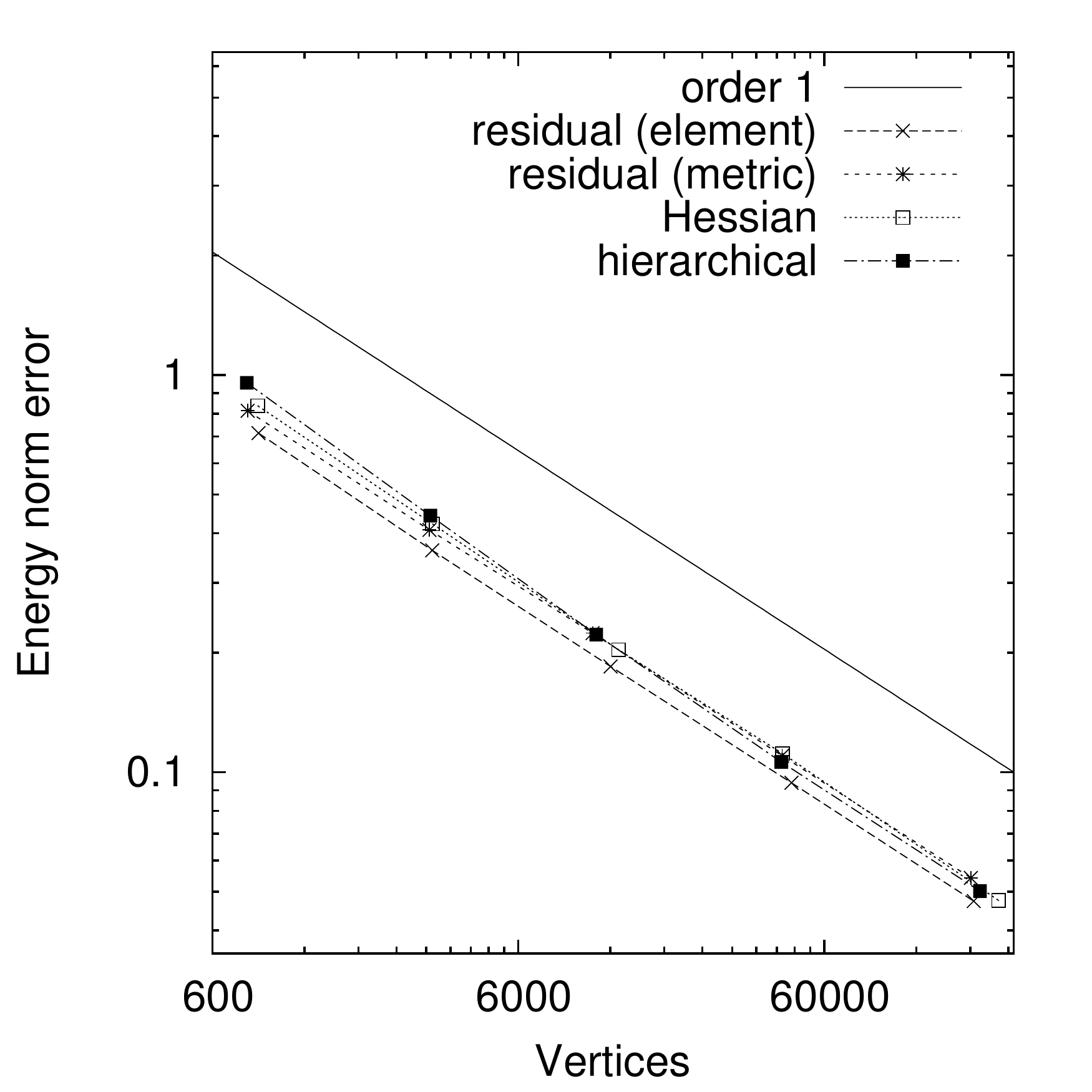}
  }
  \subfloat{
    \label{wave1-energy}
    \includegraphics[width=0.50\textwidth]{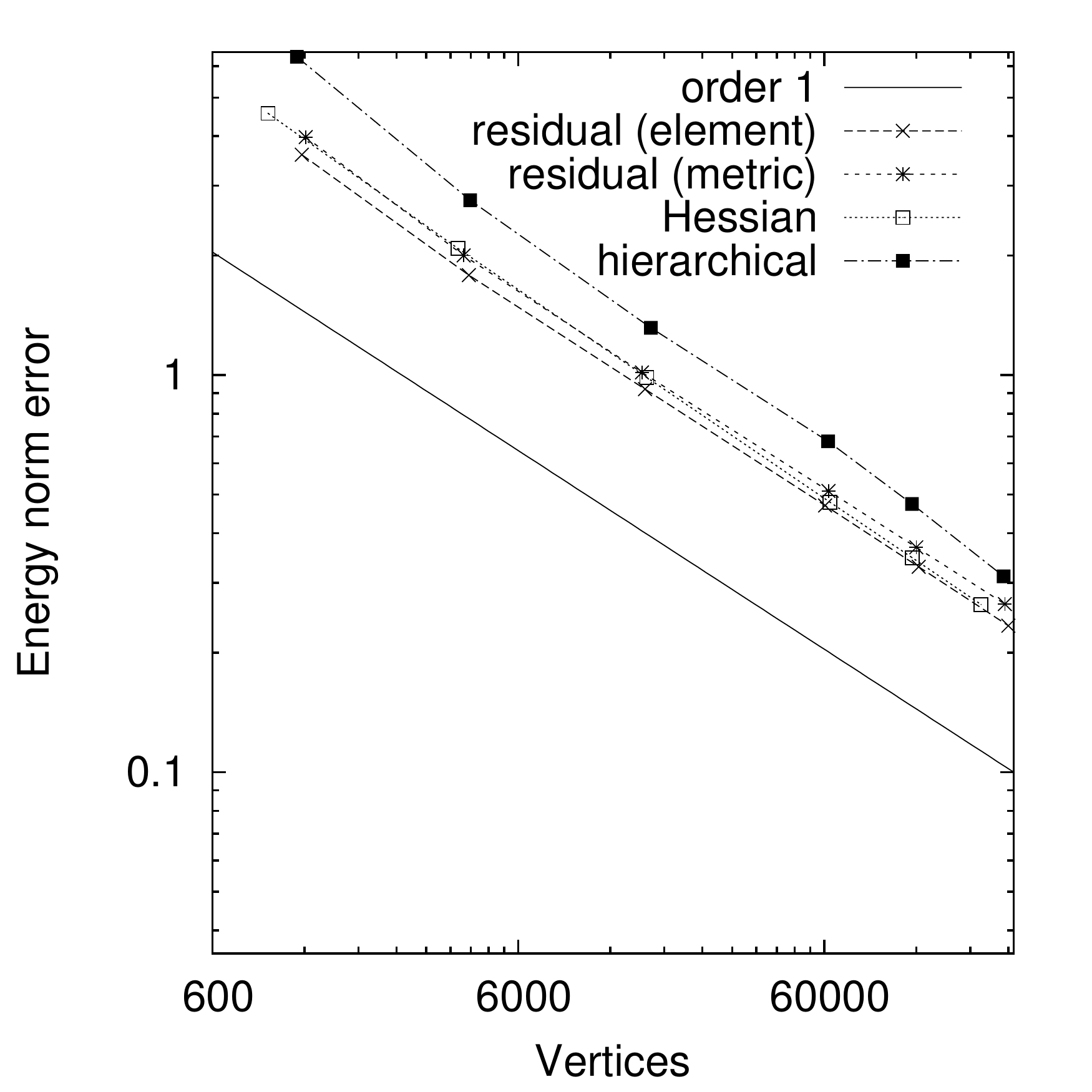}
  }
  \caption{Energy norm error calculations for $u_2$ with $\alpha=100$ (left)
    and $\alpha=1000$ (right).}
  \label{wave-both-energy}
\end{figure}
\begin{figure}[h]
  \centering
  \subfloat{
    \label{wave2-L2}
    \includegraphics[width=0.50\textwidth]{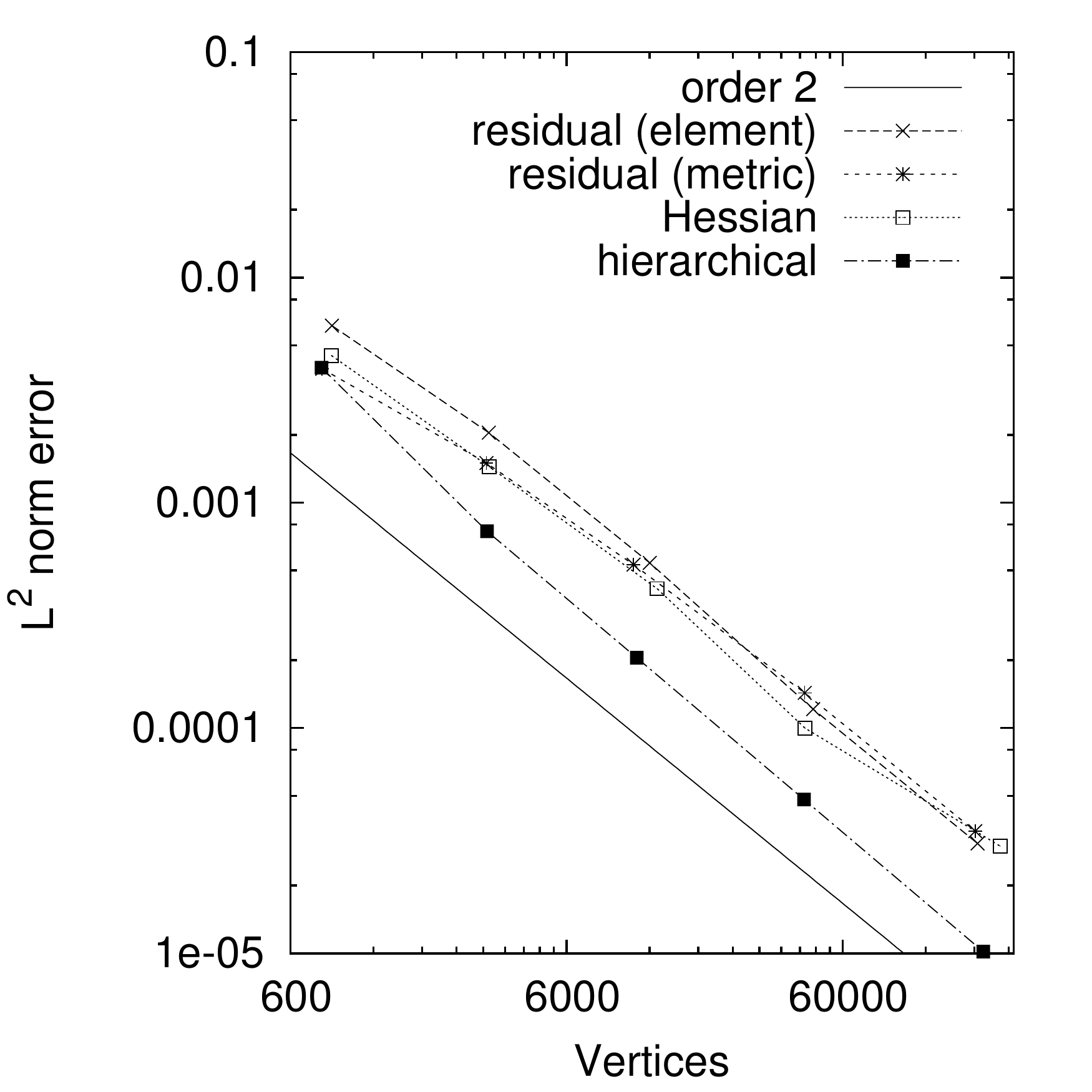}
  }
  \subfloat{
    \label{wave1-L2}
    \includegraphics[width=0.50\textwidth]{{{wave1/L2_envelope}}}
  }
  \caption{$L^2$ norm error calculations for $u_2$ with $\alpha=100$ (left) and
    $\alpha=1000$ (right). The plot on the right depicts the envelope of the
    oscillating error.}
  \label{wave-both-l2}
\end{figure}


In Figure \ref{fig:wave1-errdist}, we illustrate that the oscillation is due to
a few outlier elements, appearing just before and after the wave-front. These
elements account for a significant percentage of the overall error, and at each
iteration, slight variations in this region cause significant fluctuation in the
error. From the Figure \ref{wave1-L2} we see that for coarse meshes, this
instability arises for all methods. For the hierarchical method, as we decrease
the target error, the region is refined, and the $L^2$ error stabilizes. The
lack of stability for the $L^2$ error in the case of the residual estimator is
the result of two combined factors. First, the estimator does not detect the
fact that the $L^2$ error is still quite large outside the wave-front, and
therefore, even at very fine meshes of over $250000$ vertices, the mesh is not
refined in those regions. This observation fits within the context of
Proposition \ref{prop-l2} very well, because while we have equidistributed the
$H^1$ seminorm error over the elements, the value of $\lambda_{2,K}$ is much
smaller for elements at the wave-front, which predicts that the $L^2$ error
should also be much lower. The other contributing factor is that the mesh is not
completely stationary in this region, so that slight variations in the mesh,
which barely registered as far as the $H^1$ seminorm is concerned, cause large
variations in the $L^2$ error.

\begin{figure}[h]
  \centering 
  \subfloat{
    \label{figure-wave1-errdist}
    \includegraphics[height=0.2\textheight]{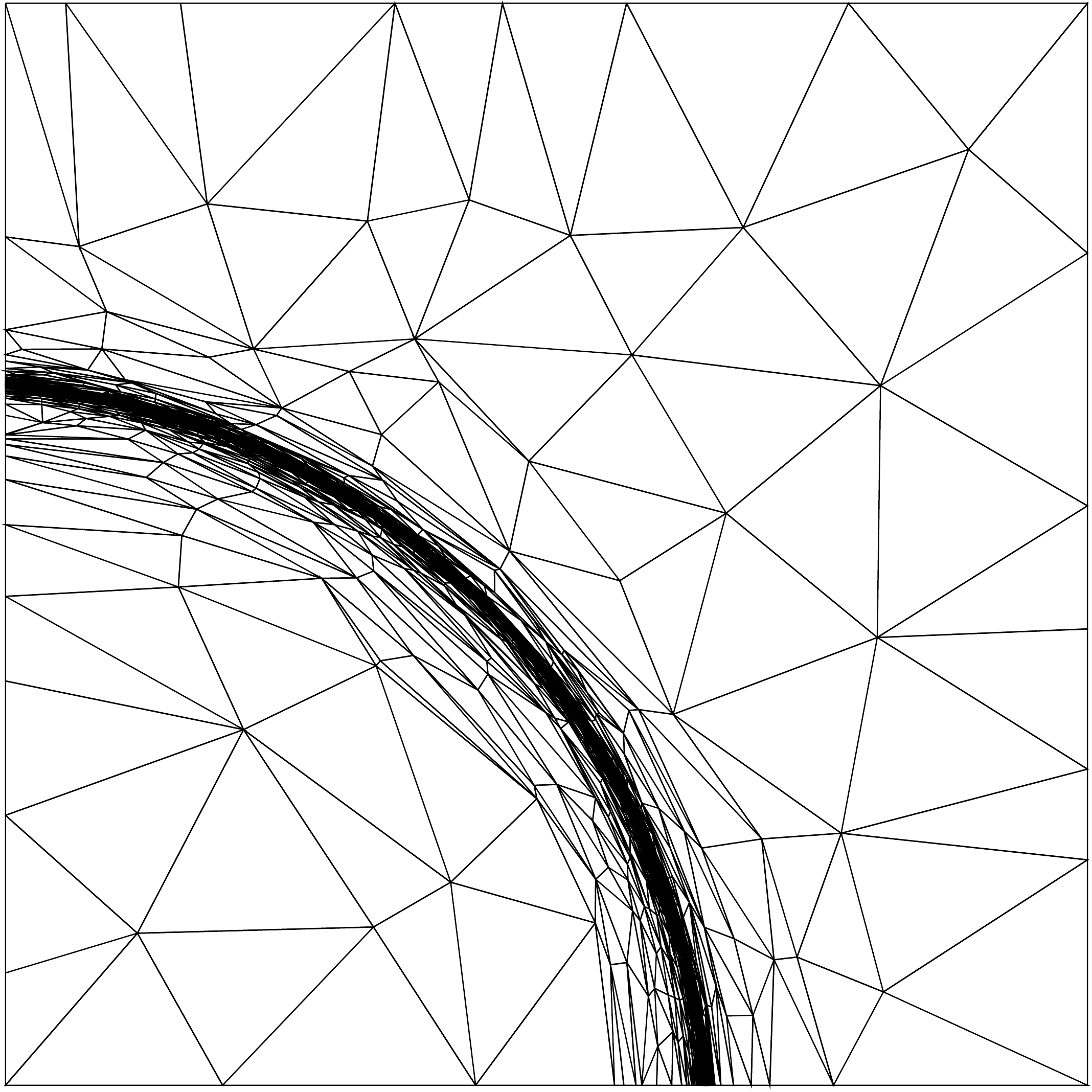}
  }
  ~
  \subfloat{
    \label{wave1-errdist-hier}
    \includegraphics[height=0.2\textheight]{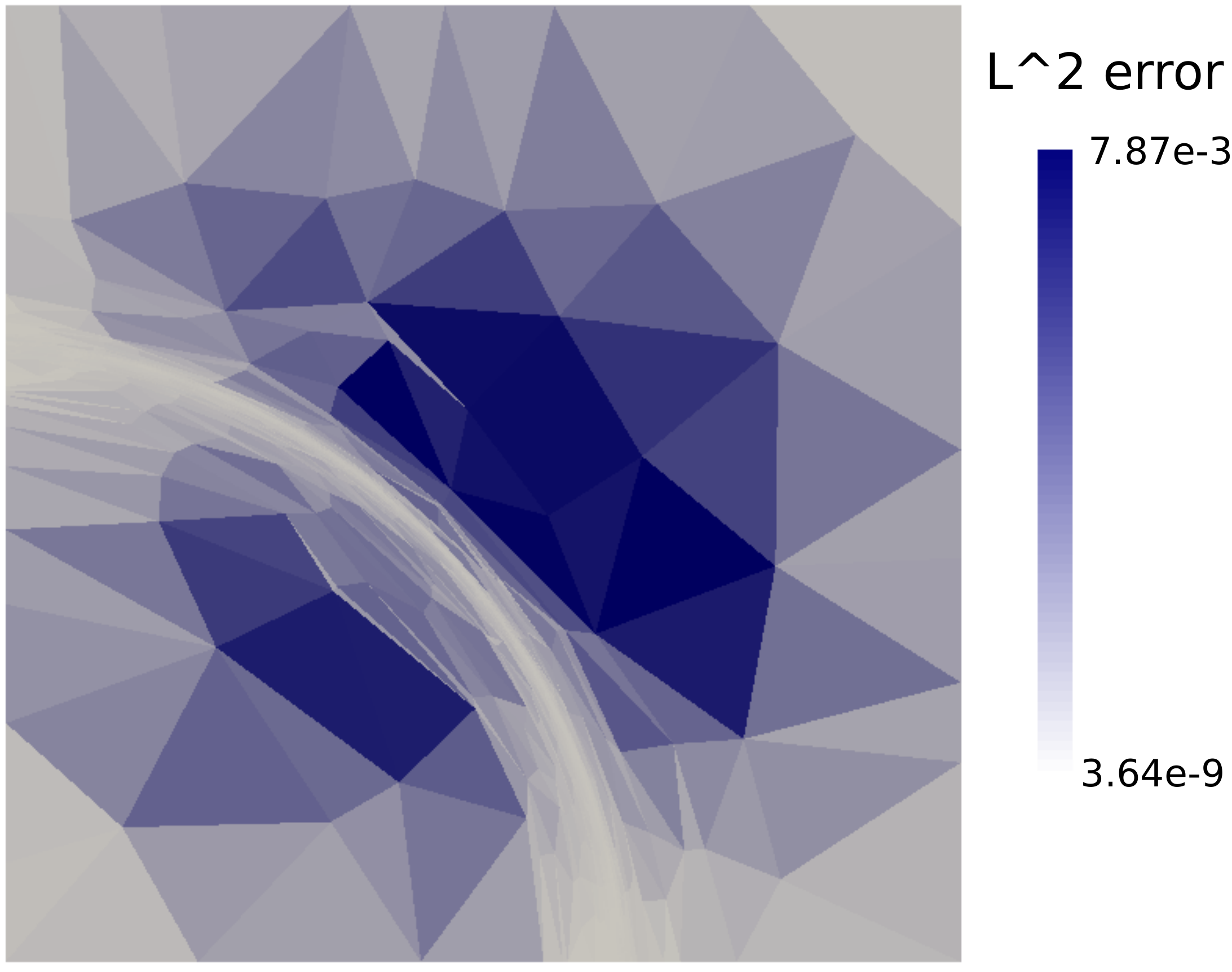}
  } 
  \caption{Mesh with about $1100$ vertices adapted with the hierarchical method
    (left) and the distribution of the exact $L^2$ error (right).}
  \label{fig:wave1-errdist}
\end{figure}


\section{Conclusion}
\label{sec:con}


We introduced an element-based mesh adaptation method for the anisotropic a
posteriori error estimator appearing in \cite{pic03a}. The method is done by
interfacing with the hierarchical estimator driver from MEF++, as introduced in
\cite{boiforfor12}. We tested the method in numerical test cases that feature
significant anisotropic behaviour and verified that the adaptation algorithm
produces anisotropic meshes and converges. Additionally, we considered an $L^2$
norm error variant of the estimator, which, under some hypotheses on the mesh,
is equivalent to the exact $L^2$ error. Numerical examples were provided to
confirm the equivalence with the exact error. Examples of adapted meshes using
the modified estimator were provided, and it was found to give improved
performance for control of the $L^2$ error over the original estimator.

The new element-based method was compared with three existing anisotropic mesh
adaptation methods for $P_1$ finite elements: residual metric based, Hessian
metric, and hierarchical. In terms of controlling the level of error with
respect to degree of freedom, the new method generally performed slightly better
for the energy norm, while the hierarchical method performed significantly
better than the other methods for the $L^2$ norm. However, the new method is
significantly more expensive from a computational standpoint. We note that the
results for both element and metric based methods for the residual estimator
were generally very close for both norms. Given the results presented in Section
\ref{sec:numer}, it seems likely that the method that obtains a given level of
error in the energy norm in the shortest time would be the residual metric
method, while for $L^2$ error it would be one of the residual metric or Hessian
methods.

Currently the authors are working on optimizing the computational aspects of the
method to make it more competitive with the other methods in terms of CPU
efficiency. Additionally, an investigation is being made to determine why the
element residual cannot be dropped from the computation, as for instance in
\cite{burpic03}.

\section*{Acknowledgements}

The authors would like to acknowledge the financial support of an Ontario
Graduate Scholarship (OGS) and by Discovery Grants of the Natural Sciences and
Engineering Research Council of Canada (NSERC).  The authors wish to thank the
professionals and researchers at GIREF, in particular Thomas Briffard, {\'E}ric
Chamberland, Andr{\'e} Fortin, and Cristian Tibirna, for making available
their code MEF++ and for their assistance in using the code during visits to the
laboratory at Universit\'e Laval and email correspondences.


\bibliographystyle{plain}
\bibliography{../../bib/list}

\end{document}